\documentclass[11pt]{amsart}

\usepackage{amsmath}
\usepackage{amssymb}
\usepackage{amsthm}
\usepackage[backref=page]{hyperref}
\usepackage[bibtex-style]{amsrefs}
\usepackage[capitalize]{cleveref}
\usepackage{enumitem}
\usepackage{subcaption}
\usepackage{todonotes}
\usepackage{soul}

\usepackage{xcolor}
\usepackage{tikz}
\usepackage{thmtools}

\theoremstyle{plain}
\newtheorem{theorem}{Theorem}[section]
\newtheorem{proposition}[theorem]{Proposition}
\newtheorem{lemma}[theorem]{Lemma}
\newtheorem{corollary}[theorem]{Corollary}
\newtheorem{conjecture}[theorem]{Conjecture}

\theoremstyle{definition}
\newtheorem{definition}[theorem]{Definition}
\newtheorem{example}[theorem]{Example}
\newtheorem{remark}[theorem]{Remark}
\newtheorem{claim}{Claim}[theorem]
\newtheorem{subclaim}{}[claim]

\newcommand{\dash}{\nobreakdash-\hspace{0pt}}
\newcommand{\del}{\backslash}
\newcommand{\ZZ}{\mathsf{Z}}
\newcommand{\Bb}{\mathcal{B}}
\newcommand{\cB}{\mathcal{B}}
\newcommand{\bL}{\mathsf{L}}
\newcommand{\bU}{\mathsf{U}}
\newcommand{\bS}{\mathsf{S}}
\newcommand{\ab}{\textbf{ab}}
\newcommand{\comout}[1]{}
\newcommand{\arc}[1]{\ensuremath{\mathbf{#1}}}
\newcommand{\identity}[1]{\ensuremath{1_{#1}}}

\usepackage{comment}

\title[Excluded minors for biased graphs]{The excluded minors for $\ZZ_{3}$-gainable and regular biased graphs}

\author[Brettell]{Nick Brettell}

\author[Campbell]{Rutger Campbell}

\author[Funk]{Daryl Funk}

\author[Mayhew]{Dillon Mayhew}

\begin{document}

\begin{abstract}
We prove that a biased graph is gainable over the group~$\ZZ_{3}$ if and only if it contains no minor isomorphic to $(4K_{2},\emptyset)$, $\pm K_{3}$, or $-K_{4}$.
We develop a theory of ``partial groups" that is analogous to that of partial fields, and we use this theory to show that a biased graph is gainable over every non-trivial group if and only if it is gainable over $\ZZ_{2}$ and $\ZZ_{3}$.
 From this we derive an independent proof of the theorem due to Gerards that a biased graph is gainable over every non-trivial group if and only if it has no minor isomorphic to $(3K_{2},\emptyset)$, $\pm K_{3}$, or $-K_{4}$.
\end{abstract}

\maketitle

\section{Introduction}

Several famous theorems in graph theory and matroid theory characterise minor-closed classes by determining their excluded minors \cites{MR1045920, MR0227041, MR1050503, MR1339849, MR0101526, MR1513158, MR2821553, MR4633709, MR0986875}.
In the matroid context, excluded-minor characterisations for the class of matroids representable over some field are especially celebrated \cites{MR0101526, MR1769191, MR0532587, MR0532586}.
In contrast, there are few excluded-minor characterisations in the context of \emph{biased graphs}.

A biased graph is a pair $(G,\cB)$ where $G$ is a graph and $\cB$ is a collection of \emph{balanced} cycles of $G$, the only rule being that a theta subgraph of $G$ is forbidden from containing precisely two balanced cycles (a \emph{theta subgraph} is the union of three internally disjoint paths linking a pair of distinct vertices).

A prototypical example of biased graphs arises from graphs equipped with a function mapping orientated edges to a group, with opposite orientations assigned inverses of one another.
In this case, a cycle is defined to be balanced if and only if it contains a simple closed walk where the product of group assignments around that walk is the identity.
It is easy to verify that the theta property holds under this definition, and in fact, Zaslavsky defines the theta property as an abstraction of graphs labelled with group elements~\cite{MR1007712}.
A biased graph arising in this way is said to be \emph{gainable} over the group. 

A graph embedded in a surface provides another example of a biased graph, where $\mathcal B$ is the collection of contractible cycles.
There is an intrinsic connection between topology and gainability: given a biased graph $(G,\cB)$, let $\mathcal K$ be the 2-cell complex obtained by adding to $G$ a disc with boundary~$C$ for each cycle $C \in \mathcal B$.
Then there exists a group over which $(G,\Bb)$ is gainable if and only if $(G,\Bb)$ is gainable over the fundamental group of $\mathcal K$; moreover, if $(G,\Bb)$ is gainable over a particular group $\Gamma$ then $\Gamma$ is a quotient of the fundamental group of $\mathcal K$ \cite{Funk15}.

The class of biased graphs that are gainable over a particular group is analogous to the class of matroids representable over a particular field: in both cases, the structure of the discrete object is encoded by algebraic labels.
The minor order of biased graphs is analogous to the minor order of graphs or matroids.
The biased graphs gainable over a group form a natural example of a minor-closed class~\cite{MR1007712}*{Corollary~5.7} (see also~\cref{prop:gaining-minor}).
So we are motivated to ask, for a group $\Gamma$, what are the excluded minors for the class of $\Gamma$\dash gainable biased graphs?
Up until now, this question has been answered only once: Zaslavsky characterised the class of 
biased graphs gainable over the cyclic group of order two, showing that there is a single excluded minor for the class.
When $G$ is a graph and $n$ is a positive integer, we let $nG$ denote the graph obtained from $G$ by replacing each non-loop edge by a parallel class of $n$ edges.
We denote the cyclic group of order $n$ by $\ZZ_n$.
Zaslavsky proved the following in 1981:

\begin{theorem}[\cite{Zas81}*{Theorem~6}]
\label{thm:Zas81}
The only excluded minor for the class of $\ZZ_{2}$\dash gainable biased graphs is $(3K_{2},\emptyset)$.
\end{theorem}

We find the excluded minors in the natural next case: the class of $\ZZ_{3}$\dash gainable biased graphs.
Let $G$ be a graph.
We write $-G$ to denote the biased graph $(G, \cB)$ where $\cB$ consists of the collection of even-length cycles of $G$.
We obtain a biased graph that we denote $\pm G$ by replacing each edge of $G$ with a parallel class of two edges, one positive and one negative, where a cycle is balanced if and only if it contains an even number of negative edges.

\begin{theorem}
\label{thm:Z3-excluded-minors}
The only excluded minors for the class of $\ZZ_{3}$\dash gainable biased graphs are $(4K_{2},\emptyset)$, $\pm K_{3}$, and $-K_{4}$.
\end{theorem}
    
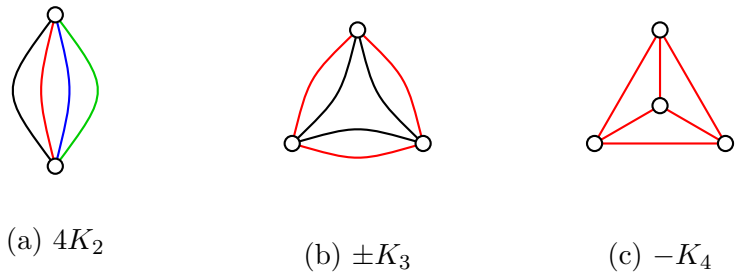
\begin{figure}[htb]

\centering

\begin{tikzpicture}[nodestyle/.style = {circle, inner sep = 0mm, minimum size = 2mm, draw = black, thick, fill = white}]

\begin{scope}[xshift = -4cm]
    
\node (bottom) at (0,-1) {};
\node (top) at (0,1) {};

\path (-1.25,-1.25)--(1.25,-1.25)--(1.25,1.25)--(-1.25,1.25)--(-1.25,-1.25);

\draw[black, thick] (0,-1)..controls (-.75,0)..(0,1);
\draw[red, thick] (0,-1)..controls (-.25,0)..(0,1);
\draw[blue, thick] (0,-1)..controls (.25,0)..(0,1);
\draw[green!80!black, thick] (0,-1)..controls (.75,0)..(0,1);

\node[nodestyle] at (bottom) {};
\node[nodestyle] at (top) {};

\node at (0,-2) {(a) $4K_{2}$};

\end{scope}

\begin{scope}[yshift = -2mm]

\node (top) at (0,1) {};
\node (left) at (.866,-.5) {};
\node (right) at (-.866,-.5) {};

\def \d {.25}

\draw[black, thick] (-.866,-.5)..controls (-.433+\d*.866,.25-\d*.5)..(0,1);
\draw[red, thick] (-.866,-.5)..controls (-.433-\d*.866,.25+\d*.5)..(0,1);

\draw[black, thick] (0,1)..controls (.433-\d*.866,.25-\d*.5)..(.866,-.5);
\draw[red, thick] (0,1)..controls (.433+\d*.866,.25+\d*.5)..(.866,-.5);

\draw[black, thick] (.866,-.5)..controls (0,\d-.5)..(-.866,-.5);
\draw[red, thick] (.866,-.5)..controls (0,-\d-.5)..(-.866,-.5);

\node[nodestyle] at (top) {};
\node[nodestyle] at (left) {};
\node[nodestyle] at (right) {};

\node at (0, -2) {(b) $\pm K_{3}$};

\end{scope}

\begin{scope}[xshift = 4cm, yshift = -2mm]

\node (top) at (0,1) {};
\node (left) at (.866,-.5) {};
\node (right) at (-.866,-.5) {};
\node (center) at (0,0) {};

\draw[red, thick] (-.866,-.5)--(0,1);
\draw[red, thick] (0,1)--(.866,-.5);
\draw[red, thick] (.866,-.5)--(-.866,-.5);
\draw[red, thick] (0,0)--(0,1);
\draw[red, thick] (0,0)--(.866,-.5);
\draw[red, thick] (0,0)--(-.866,-.5);

\node[nodestyle] at (top) {};
\node[nodestyle] at (left) {};
\node[nodestyle] at (right) {};
\node[nodestyle] at (center) {};

\node at (0, -2) {(c) $-K_{4}$};

\end{scope}

\end{tikzpicture}

\caption{The excluded minors for $\ZZ_{3}$\dash gainable biased graphs, where a cycle is balanced if and only if it contains an even number of edges of each (non-black) colour.
}
\label{fig:Z3-excluded-minors}
\end{figure}

Rota famously conjectured that when $\mathbb{F}$ is a finite field, the class of $\mathbb{F}$\dash representable matroids has only finitely many excluded minors~\cite{MR0505646}.
We conjecture that when $\Gamma$ is a finite group, the class of $\Gamma$\dash gainable biased graphs has only finitely many excluded minors.
\cref{thm:Zas81,thm:Z3-excluded-minors} correspond to the first two non-trivial cases of this conjecture.
In fact, we make the following stronger conjecture, which combines the biased-graph variant of Rota's Conjecture with the statement that $\Gamma$\dash gainable biased graphs are well-quasi-ordered when $\Gamma$ is finite.

\begin{conjecture}
\label{biased-Rota}
Let $\Gamma$ be a finite group and let $\mathcal{G}$ be a minor-closed class of $\Gamma$\dash gainable biased graphs.
There are only finitely many biased graphs that are excluded minors for $\mathcal{G}$.
\end{conjecture}

Geelen, Gerards, and Whittle have announced a structural theorem for minor-closed class of $\mathbb{F}$\dash representable matroids (when $\mathbb{F}$ is a finite field) that generalises the structural theorem of Robertson and Seymour on minor-closed classes of graphs~\cite{MR3184116}.
\cref{biased-Rota} is a significant generalisation of the Graph Minor Theorem in another direction --- with $\Gamma$ equal to the trivial group, Conjecture~\ref{biased-Rota} would imply the result that any minor-closed class of graphs has only finitely many excluded minors.
Conjecture~\ref{biased-Rota} fails when $\Gamma$ is not finite, because in this case there are infinitely many excluded minors for the class of $\Gamma$\dash gainable biased graphs~\cite{Funk15}*{Theorem~2.12}.

Assume that $\mathcal{H}$ is a set of at least two groups.
One could also attempt to find an excluded-minor characterisation of the biased graphs that are gainable over every group in $\mathcal{H}$.
The only such characterisation known so far is due to Gerards~\cite{MR1106635}, who found the excluded minors for the biased graphs that are gainable over every non-trivial group.  We say that a biased graph is \emph{regular} if it is gainable over every non-trivial group.  By developing a notion of ``partial groups'' (analogous to partial fields, as used in matroid representation theory~\cites{MR1390574,MR2718674,MR2563513}), we prove the following:

\begin{restatable}{theorem}{mainone}
    \label{thm:main1}
    Let $\Omega$ be a biased graph.
    Then $\Omega$ is regular if and only if $\Omega$ is $\ZZ_2$\dash gainable and $\ZZ_3$\dash gainable. 
\end{restatable}

This naturally leads to an independent proof of the following, first proved by Gerards~\cite{MR1106635}:

\begin{restatable}{theorem}{maintwo}
    \label{thm:main2}
    The excluded minors for the class of regular biased graphs are $(3K_2,\emptyset)$, $\pm K_3$, and $-K_4$.
\end{restatable}

We also show the following:
\begin{theorem}
    \label{thm:main3}
    A biased graph $\Omega$ is gainable over all groups of size at least three if and only if $\Omega$ is $\ZZ_3$\dash gainable and $\ZZ^2_2$\dash gainable.
\end{theorem}

\begin{theorem}
    \label{thm:main4}
    A biased graph $\Omega$ is gainable over all groups with an element of order at least three if and only if $\Omega$ is $\ZZ_3$\dash gainable and $\ZZ_4$\dash gainable.
\end{theorem}

Finding the list of excluded minors for either of the classes described in the last two theorems is an open problem.

\subsection{Biased graphs and matroids}

Each biased graph represents a family of matroids, known as \emph{quasi-graphic matroids} (see \cites{MR4037634, MR3742182}).
The best-known members of this family are the \emph{frame} and \emph{lift} matroids of a biased graph.
Let $G$ be a graph and let $(\cB, \mathcal{L}, \mathcal{F})$ be a partition (into possibly empty sets) of the collection of cycles of $G$  such that $\Omega=(G,\cB)$ is a biased graph and whenever $C_{L}$ and $C_{F}$ are cycles in $\mathcal{L}$ and $\mathcal{F}$ respectively, there is at least one vertex that is common to $C_{L}$ and $C_{F}$.
We call the tuple $(G, \cB, \mathcal{L}, \mathcal{F})$ a \emph{biased framework}.

The quasi-graphic matroid represented by $(G, \cB, \mathcal{L}, \mathcal{F})$ has $E(G)$ as its ground set and its dependent sets are those subsets of $E(G)$ that contain either a cycle in $\cB$, two cycles in $\mathcal{L}$, or are connected and contain two cycles. 
In the case that $\mathcal{L}$ is empty, this specialises to the \emph{frame matroid} of the biased graph $\Omega=(G,\cB)$, and we denote this matroid by $F(\Omega)$. When $\mathcal{F}$ is empty we obtain the \emph{lift matroid} $L(\Omega)$, which is an elementary lift of the cycle matroid of $G$.

A biased framework $(G, \cB, \mathcal{L}, \mathcal{F})$ is \emph{$\Gamma$\dash gainable} if $(G,\cB)$ is a $\Gamma$\dash gainable biased graph. 
The following is a matroidal analogue of \cref{biased-Rota}.

\begin{conjecture}
\label{frame-lift-Rota}
Let $\Gamma$ be a finite group and let $\mathcal{M}$ be a minor-closed class of matroids, each of which is represented by a $\Gamma$\dash gainable biased framework.
There are only finitely many matroids that are excluded minors for $\mathcal{M}$.
\end{conjecture}

Equivalently, we conjecture that the class of quasi-graphic matroids that are $\Gamma$-gainable has finitely many excluded minors and does not contain an infinite pairwise-incomparable collection.
If we restrict \cref{frame-lift-Rota} to the case when $\mathcal{M}$ contains only frame matroids, then we obtain Conjectures~1.3 and 1.4 in \cite{MR4395073}.
The finiteness of $\Gamma$ is necessary in \cref{frame-lift-Rota}, because if $\Gamma$ is infinite, the class of frame matroids of $\Gamma$\dash gainable biased graphs has infinitely many excluded minors \cite{Funk15}*{Theorem~2.19}.

\begin{remark}
Chen and Geelen showed that there are infinitely many excluded minors for both lift and frame matroids \cite{CHEN201846}.
These excluded minors are all quasi-graphic matroids.
However, every one of these excluded minors contains one of two $8$\dash element matroids as a minor (one for the frame excluded minors and one for the lift excluded minors).
Neither one of these matroids can be represented by a $\Gamma$\dash gainable biased framework, for any group~$\Gamma$.
Thus Chen and Geelen's result does not disprove \cref{frame-lift-Rota}.
\end{remark}

Let $\Omega$ be a biased graph and let $\mathbb{F}$ be a field.
If $\Omega$ is gainable over a subgroup of the multiplicative group $\mathbb{F}^{*}$, then $F(\Omega)$ is $\mathbb{F}$\dash representable \cite{MR2017726}*{Theorem~2.1}.
Similarly, if $\Omega$ is gainable over a subgroup of the additive group $\mathbb{F}^{+}$, then $L(\Omega)$ is $\mathbb{F}$\dash representable \cite{MR2017726}*{Theorem~4.1}.
Conversely, if $\Omega$ is a biased framework for a 3-connected quasi-graphic $\mathbb{F}$\dash representable matroid, then $\Omega$ is either $\mathbb{F}^{*}$\dash gainable or $\mathbb{F}^{+}$\dash gainable \cite{MR4392273}*{Theorem~5.5}.

This shows that if $\Omega$ is $\ZZ_{3}$\dash gainable, then $F(\Omega)$ is $\mathbb{F}_{4}$\dash representable and $L(\Omega)$ is $\mathbb{F}_{3}$\dash representable.
Given that the classes of $\mathbb{F}_{3}$\dash representable and $\mathbb{F}_{4}$\dash representable matroids have known excluded-minor characterisations, one might wonder whether it is possible to use these to obtain \cref{thm:Z3-excluded-minors}.
This seems not to be the case.
The underlying reason is that a frame or lift matroid may be represented by multiple different biased graphs, and it may be that only some of these biased graphs are gainable over a particular group.
To see this, let $\Omega_1$ be $(3K_{2},\emptyset)$ and let $\Omega_2$ be the biased graph with underlying graph $K_{3}$ where the unique cycle is balanced.
Then $F(\Omega_1)=F(\Omega_2)$ and $L(\Omega_1)=L(\Omega_2)$.
Furthermore, $\Omega_2$ is gainable over every group but $\Omega_1$ is not $\ZZ_{2}$\dash gainable.
So in the case that a frame matroid $F(\Omega)$ is an excluded minor for $\mathbb{F}$\dash representable matroids, we \emph{can} conclude that $\Omega$ is not $\mathbb{F}^{*}$\dash gainable.
However, when $\Omega'$ is a proper minor of $\Omega$, we do not know that $\Omega'$ is $\mathbb{F}^{*}$\dash gainable, despite the fact that $F(\Omega')$ is $\mathbb{F}$\dash representable.
Therefore $\Omega$ may not be an excluded minor for $\mathbb{F}^{*}$\dash gainability.
To make this concrete, $U_{2,5}$ is a frame matroid and an excluded minor for $\mathbb{F}_{3}$\dash representability, but there is no biased graph $\Omega$ that is an excluded minor for $\mathbb{F}_{3}^{*}$\dash gainability where $F(\Omega)=U_{2,5}$.

In the other direction, if $\Omega$ is an excluded minor for $\mathbb{F}^{*}$\dash gainability, it may be that the matroid $F(\Omega)$ is $\mathbb{F}$\dash representable.
We can see this by referring to \cref{thm:Zas81}: $(3K_{2},\emptyset)$ is an excluded minor for $\mathbb{F}_{3}^{*}$\dash gainable biased graphs, but the frame matroid of $(3K_{2},\emptyset)$ is the uniform matroid $U_{2,3}$, and this is $\mathbb{F}_{3}$\dash representable.

The fact that $\Gamma$\dash gainability is not an invariant amongst the biased graphs representing a given quasi-graphic matroid also means that a resolution of \cref{biased-Rota} is unlikely to help with proving \cref{frame-lift-Rota}, and vice versa.

Even if it were possible to use the excluded-minor characterisation of $\mathbb{F}$\dash representable matroids to obtain excluded-minor characterisations of biased graphs gainable over $\mathbb{F}^{*}$ or $\mathbb{F}^{+}$, this would be of no help for almost all groups, since they are not subgroups of multiplicative or additive groups of fields.
Our hope is that the techniques we develop here will also be of use in obtaining further excluded-minor characterisations of classes of biased graphs in the future, so for these reasons, our arguments do not contain any references to matroids.

\section{Preliminaries}

Groups are written multiplicatively; we denote the identity of the group~$\Gamma$ by $\identity{\Gamma}$.
For a positive integer $m$, we write $\ZZ_m$ to denote the cyclic group of order $m$ consisting of elements $\{1,\omega,\dots,\omega^{m-1}\}$ under multiplication, where $\omega^m=1$.
In particular, the trivial group is $\ZZ_1=\{1\}$.
We write $\ZZ_m^n$ to denote the direct product of $n$ copies of $\ZZ_m$; for example, $\ZZ^2_2 = \ZZ_2 \times \ZZ_2$ is the Klein four-group.

\subsection{Graphs.}
Graphs may contains loops and parallel edges.
A non-loop edge $e$ is \emph{pendant} if there is a vertex that is incident only with $e$ and no other edge.
A vertex $v$ is \emph{isolated} if there are no edges incident with $v$.
A \emph{walk} is an alternating sequence $u_{0},e_{0},u_{1},\ldots, u_{n-1}, e_{n-1}, u_{n}$ of vertices and edges such that the set of vertices incident with each edge $e_{i}$ is equal to $\{u_{i}, u_{i+1}\}$.
We also consider a single vertex to be a walk.
A \emph{path} is a walk in which the vertices are pairwise distinct.
Two paths from $u$ to $v$ are \emph{internally disjoint} if the only vertices they have in common are $u$ and $v$.
If $n>0$ and $u_{0}=u_{n}$ while $u_{0},u_{1},\ldots, u_{n-1}$ are pairwise disjoint, then we say the walk is \emph{simple} and \emph{closed}.
If $G$ is a graph and $X$ is a set of edges, then $G[X]$ is the subgraph of $G$ with $X$ as its edge-set, where the vertices of $G[X]$ are exactly the vertices of $G$ that are incident with at least one edge in $X$.
A \emph{cycle} is a subgraph $G[X]$ where $X$ is a non-empty set of edges such that $G[X]$ is connected and regular with degree $2$.
When the meaning is clear, we may write $X$ to describe both a set of edges and the subgraph $G[X]$.
In the same way, we may think of a spanning tree or a cycle as being either a subgraph or a set of edges.

If $X$ is a set of vertices of $G$, then $G[X]$ is the subgraph with $X$ as its vertex-set, where the edges of $G[X]$ are exactly the edges of $G$ that are incident only with vertices in $X$.
If $U$ is a set of vertices or a subgraph and $v$ is a vertex not in $U$, then a \emph{$v$\dash $U$ path} is a minimal path that has $v$ and a vertex in $U$ as end-vertices.

A \emph{separation} of a graph $G$ is a pair of edge-disjoint subgraphs $(X,Y)$ of $G$ whose union is $G$, where $|E(X)|,|E(Y)| \ge 1$. 
The \emph{order} of the separation is $|V(X) \cap V(Y)|$. We call a separation of order $k$ a \emph{$k$\dash separation}. 
If $V(X)-V(Y)$ and $V(Y)-V(X)$ are both non-empty, then the separation is \emph{proper}. 
A graph with no isolated vertices is \emph{$k$\dash connected} if it has at least $k+1$ vertices and no proper separation of order less than $k$.

Let $(X,Y)$ be a $1$\dash separation of the graph $G$, and let $v$ be the vertex in $V(X)\cap V(Y)$.
If $(X,Y)$ is proper, then $v$ is a \emph{cut-vertex} of $G$.
Note that $(X,Y)$ fails to be proper if and only if either $X$ or $Y$ consists of loop edges incident with $v$.
A graph $G$ is \emph{biconnected} if it is connected and has no $1$\dash separations (either proper or non-proper).
Therefore a biconnected graph with at least two edges is loopless. 
A \emph{block} of $G$ is a maximal biconnected subgraph.

The next result follows from~\cite{MR0210617}*{9.65}.

\begin{proposition}
\label{prop:delete-or-contract}
Let $G$ be a biconnected graph and let $x$ be an edge of $G$.
Then either $G\backslash x$ or $G/x$ is biconnected.
\end{proposition}

A graph $G$ is \emph{minimally $2$\dash connected} if it is $2$\dash connected but whenever $e$ is an edge of $G$ we have that $G \backslash e$ is not $2$\dash connected.
Note that a minimally $2$\dash connected graph cannot contain loops or parallel edges.
We will use the following:

\begin{lemma}[Dirac \cite{MR216975}*{(6),(5)}, see also Mader \cite{MR561307}]
\label{diraclemma}
    A minimally $2$\dash connected graph $G$ has at least $(|V(G)|+4)/3$ vertices of degree two.
\end{lemma}

\begin{corollary}
    \label{helper}
    Let $G$ be a $2$\dash connected graph 
    such that every vertex has degree at least three.
    Then there exist distinct $e,f \in E(G)$ such that each of $G \backslash e$, $G \backslash f$, $G \backslash \{e,f\}$ is $2$\dash connected.
\end{corollary}
\begin{proof}
The graph $G$ has at least $3$ vertices, by the definition of $2$\dash connectivity, and no vertices of degree $2$.  Thus, by \cref{diraclemma}, it is not minimally $2$\dash connected. So there exists $e \in E(G)$ such that $G \backslash e$ is $2$\dash connected.
Then $G \backslash e$ has at most two vertices of degree 2, so \cref{diraclemma} implies that if $G \backslash e$ is minimally $2$\dash connected, then $|V(G \backslash e)| \le 2$.
This implies $|V(G)| \le 2$, a contradiction.
Thus $G \backslash e$ is not minimally $2$\dash connected, implying there is an edge $f$ such that $G \backslash \{e,f\}$ is $2$\dash connected.
It follows that $G \backslash f$ is also $2$\dash connected.
\end{proof}

\subsection{Biased graphs.}

Let $G$ be a graph.
A \emph{theta subgraph} is a subgraph consisting of two distinct vertices $u$ and $v$ and three internally disjoint paths joining $u$ and $v$.
Let $\cB$ be a set of cycles.
If no theta subgraph of $G$ contains exactly two cycles in $\cB$, then $\cB$ satisfies the \emph{theta property}.
In this case, $(G,\cB)$ is a \emph{biased graph}.
The cycles in $\cB$ are said to be \emph{balanced}.
If $X$ is a set of edges or a subgraph, then $X$ is \emph{balanced} if every cycle of $X$ is in $\cB$.
Otherwise, $X$ is \emph{unbalanced}.
We say that the \emph{bias} of a cycle is its status as balanced or unbalanced.
So saying that two cycles have the same bias means that both belong to $\Bb$ or that neither one does.
Let $\Omega=(G,\cB)$ be a biased graph.
If $A$ is a set of edges, then $\Omega[A]$ is the biased graph on the underlying graph $G[A]$, where a cycle is balanced in $\Omega[A]$ if and only if it is balanced in $\Omega$.

We frequently blur the distinction between a biased graph and its underlying graph, so we may, for example, say that a biased graph is connected.

\subsection{Balanced reroutings.}

Let $(G,\Bb)$ be a biased graph, let $P$ be a path in $G$, and let $Q$ be a path internally disjoint from $P$ linking a pair of distinct vertices $x,y$ of $P$. 
We say the path $P'$ obtained from $P$ by replacing its subpath linking $x$ and $y$ (which we denote by $xPy$) with $Q$ is obtained by \emph{rerouting} $P$ \emph{along} $Q$. 
If the cycle consisting of $Q$ and $xPy$ is in $\Bb$, then we refer to this as \emph{rerouting along a balanced cycle} or a \emph{balanced rerouting}.

Let $u,v$ be vertices of $G$. 
Observe that given any pair of $u$\dash $v$ paths in $G$, one may be transformed to the other via some sequence of reroutings. 
Let $C$ be a cycle and let $P$ be a path contained in $C$, having distinct end-vertices $x,y$.
Let $Q$ be an $x$\dash $y$ path internally disjoint from $C$, and assume the cycle $P \cup Q$ is balanced. 
Then the balanced rerouting of $P$ along $Q$ yields a new cycle $C'$. 
Since $C \cup Q$ is a theta and the cycle $P \cup Q$ is balanced, the theta property says that $C$ and $C'$ are either both balanced or both unbalanced. 
We say that $C'$ is obtained from $C$ by \emph{balanced rerouting}.
Assume that $P$, $Q_{1}$, and $Q_{2}$ are paths joining the distinct vertices $x$ and $y$, where $P$ is internally disjoint from both $Q_{1}$ and $Q_{2}$.
Thus $P\cup Q_{1}$ and $P\cup Q_{2}$ are both cycles.
If there is some balanced subgraph that contains the union of $Q_{1}$ and $Q_{2}$, then balanced rerouting shows that the union $P \cup Q_{1}$ has the same bias as $P\cup Q_{2}$.

\subsection{Gainings.}

Let $G$ be a graph.
An \emph{arc} of $G$ is a triple of the form $(e,u,v)$, where $e$ is an edge, $u$ and $v$ are vertices incident with $e$, and where $u\ne v$ if $e$ is a non-loop edge.
If $\arc{a}=(e,u,v)$ is an arc, then $\arc{a}^{-1}$ is the arc $(e,v,u)$.
If $e$ is an edge, then an \emph{arc of $e$} is an arc of the form $(e,u,v)$, where $u$ and $v$ are vertices incident with $e$.
Let $A(G)$ denote the set of arcs of $G$.
Let $\Gamma$ be a group.
A \emph{$\Gamma$\dash gaining} of $G$ is a function $\gamma \colon A(G)\to \Gamma$ satisfying $\gamma(\arc{a}^{-1})=\gamma(\arc{a})^{-1}$ whenever $\arc{a}$ is an arc of a non-loop edge.
 
When $e$ is an edge of $G$ with distinct endpoints $u,v$ and $\gamma(e,u,v) = \gamma(e,v,u) = h \in \Gamma$, then we say for short that $\gamma$ assigns $h$ to the edge $e$.
This occurs when $h$ is an element of order two or when $h$ is the identity element of $\Gamma$. 
Thus in particular we may speak of edges being assigned the identity element of $\Gamma$. 

Assume that $W=u_{0},e_{0},u_{1},e_{1},\ldots, e_{n-1},u_{n}$ is a walk in $G$.
The \emph{gain} of $W$ (relative to $\gamma$) is
\[
\gamma(W)=\prod_{\arc{a}\in W}\gamma(a)=\gamma(\arc{a}_{0})\gamma(\arc{a}_{1})\cdots\gamma(\arc{a}_{n-1}),
\]
where $a_i = (e_i, u_i, u_{i+1})$ and we take the product in the order given by the walk.   
Let $W$, $W'$ be two simple closed walks contained in a cycle $C$. Then either $\gamma(W)$ and $\gamma(W')$ are conjugate or $\gamma(W)$ and $\gamma(W')^{-1}$ are conjugate. Thus $\gamma(W)=\identity{\Gamma}$ if and only if $\gamma(W')=\identity{\Gamma}$. When $\gamma(W)=\identity{\Gamma}$ for some (and therefore all) simple closed walks in a cycle $C$, we say the gain of $C$ is the identity.

Let $\gamma$ be a $\Gamma$\dash gaining of the graph $G$.
Let $\cB_{\gamma}$ be the set of cycles with gain equal to $\identity{\Gamma}$.
It is straightforward to check that $\cB_{\gamma}$ satisfies the theta property.
Let $\Omega=(G,\cB)$ be a biased graph.
We say that $\Omega$ is \emph{$\Gamma$\dash gainable} if there is a $\Gamma$\dash gaining $\gamma$ such that $\cB=\cB_{\gamma}$.
In this case we say that $\gamma$ is an \emph{$\Gamma$\dash gaining} of $\Omega$.
From~\cite{MR1007712}*{Corollary~5.7} we see that if $\Gamma$ is a group, then the class of $\Gamma$\dash gainable biased graphs is closed under minors.
(We extend this result in \cref{prop:gaining-minor}.)

Let $\gamma$ be a $\Gamma$\dash gaining of the biased graph $\Omega=(G,\cB)$ and let $\eta\colon V(G) \to \Gamma$ be a function.
The gain function $\gamma^\eta$ takes any arc $\arc{a}=(e,u,v)$ to $\eta(u)^{-1}\gamma(\arc{a})\eta(v)$.
We say $\gamma^{\eta}$ is obtained from $\gamma$ by \emph{switching}.
If $v$ is a vertex and $\eta$ sends every vertex except for $v$ to $\identity{\Gamma}$, then we say that $\gamma^{\eta}$ is obtained from $\gamma$ by \emph{switching at $v$}.
It is clear that $\cB_{\gamma^{\eta}} = \cB_{\gamma}$.
If $T$ is a forest of $G$, there is a function $\eta$ such that $\gamma^{\eta}$ takes any edge in $T$ to $\identity{\Gamma}$~\cite{MR1007712}*{Lemma~5.3}.

\begin{proposition}
\label{prop:agree-on-tree}
Let $\Gamma$ be a group and let $\Omega$ be a connected biased graph.
Let $\gamma$ and $\gamma'$ be $\Gamma$\dash gainings of $\Omega$.
Let $T$ be a spanning tree of $\Omega$ and assume that $\gamma$ and $\gamma'$ both take every edge of $T$ to $\identity{\Gamma}$.
Then $\gamma$ and $\gamma'$ take the same edges to $\identity{\Gamma}$.
\end{proposition}

\begin{proof}
If this fails then we can assume without loss of generality that there is an edge $e$ and an arc \arc{a} of $e$ such that $\gamma(\arc{a}) = \identity{\Gamma}$ and $\gamma'(\arc{a})\ne\identity{\Gamma}$.
It must be the case that $e$ is not in $T$.
There is a cycle $C$ contained in the union of $T$ and $e$.
The gain of $C$ is the identity under $\gamma$, so $C$ is balanced, but its gain is not the identity under $\gamma'$, so it is not balanced.
This contradiction completes the proof.
\end{proof}

\subsection{Minors.}

Let $\Omega=(G,\cB)$ be a biased graph.
Let $A$ and $B$ be disjoint sets of edges such that the biased subgraph $\Omega[A]$ induced by $A$ is balanced. 
Let $\mathcal{P}$ be the set of connected components of $G[A]$.
We define $G/A\backslash B$ to be the graph with $\mathcal{P}$ as its vertex-set and edge-set $E(G)-(A\cup B)$, where an edge $e\in E(G)-(A\cup B)$ is incident with $P\in \mathcal{P}$ if and only if there is a vertex $v\in P$ such that $e$ is incident with $v$ in $G$.
Now we construct the set of cycles $\cB/A\backslash B$.
Let $C$ be a cycle of $G/A\backslash B$.
There is at least one cycle $C'$ in $G$ such that every edge of $C$ is in $C'$, and every edge of $C'-C$ is in $A$.
By balanced rerouting, all such cycles have the same bias, so the following definition is well-defined.
We declare $C$ to belong to $\cB/A\backslash B$ if and only if $C'$ is in $\cB$.
Now $\Omega/A\backslash B$ is the biased graph $(G/A\backslash B, \cB/A\backslash B)$.
We say that $\Omega/A\backslash B$ is obtained from $\Omega$ by \emph{contracting} $A$ and \emph{deleting} $B$.
A \emph{minor} of $\Omega$ is any biased graph obtained from a biased graph of the form $\Omega/A\backslash B$ by possibly deleting isolated vertices.
Note that in this definition of minors, we cannot contract an unbalanced loop, so the only way to get rid of a unbalanced loop is by deleting it. This is in accordance to the fact that unbalanced loops would correspond to a non-identity gaining on a loop or a noncontractible loop on a surface.

Let $\mathcal{M}$ be a minor-closed class of graphs, matroids, or biased graphs.
A minor-minimal graph, matroid, or biased graph that is not in $\mathcal{M}$ is an \emph{excluded minor} for the class.

\begin{proposition}\label{prop:unbalanced-contract}
Let $\Omega$ be a biased graph with at least one unbalanced cycle, and let $e$ be a non-loop edge of $\Omega$.
Then $\Omega/e$ contains an unbalanced cycle.
\end{proposition}

\begin{proof}
Let $C$ be an unbalanced cycle of $\Omega$.
If $e$ is an edge of $C$, then $C/e$ is an unbalanced cycle of $\Omega/e$, since $e$ is not a loop.
If $e$ is not a chord of $C$, then $C$ is an unbalanced cycle of $\Omega/e$. 
The remaining possibility is that $e$ is a chord of $C$. 
Then the union of $C$ and $e$ is a theta subgraph.
One of the cycles $D$ in this subgraph contains $e$ and is unbalanced,  by the theta property. 
The cycle $D/e$ is an unbalanced cycle in $\Omega/e$. 
\end{proof}

\subsection{The excluded minors.}

In this section we prove basic facts about excluded minors for classes of $\Gamma$\dash gainable biased graphs and we certify that the biased graphs in \cref{thm:Z3-excluded-minors} really are excluded minors for the class of $\ZZ_{3}$\dash gainable biased graphs.
The next lemma shows that an excluded minor must be biconnected (implying that it is loopless).

\begin{lemma}
    \label{exminors2conn}
    Let $\Gamma$ be a non-trivial group and let $\Omega$ be an excluded minor for the class of biased graphs that are $\Gamma$\dash gainable.  
    Then $\Omega$ has at least two vertices, has minimum degree at least three, and is biconnected.
    Furthermore $\Omega$ has no $2$\dash separation where one side is balanced and contains at least two edges (so in particular $\Omega$ has no balanced $2$\dash cycle). 
\end{lemma}

\begin{proof}
If $\Omega$ has an isolated vertex, then deleting that isolated vertex produces a proper minor, which is therefore $\Gamma$\dash gainable.
This implies that $\Omega$ is $\Gamma$\dash gainable, a contradiction.
Therefore $\Omega$ has no isolated vertex.
If $\Omega$ has only a single vertex, then it is gainable over every group.
Therefore $\Omega$ has at least two vertices.
Suppose $\Omega$ has a separation $(X,Y)$ with $|V(X) \cap V(Y)| \leq 1$.
    Recall that this means that $X$ and $Y$ are edge-disjoint subgraphs, each containing at least one edge.
    By minimality each of $X$ and $Y$ are $\Gamma$\dash gainable. 
    Let $\gamma_X$ be a $\Gamma$\dash gaining of $X$ and let $\gamma_Y$ be a $\Gamma$\dash gaining of $Y$. 
    Since no cycle of $\Omega$ contains edges from both $X$ and $Y$, the $\Gamma$\dash gain function on $\Omega$ obtained as the union of $\gamma_X$ and $\gamma_Y$ is a $\Gamma$\dash gaining of $\Omega$, a contradiction. 
Therefore no such separation can exist, so $\Omega$ is biconnected.

    Suppose $\Omega$ has a vertex $v$ of degree less than three.
    It cannot have a degree-one vertex, or else it would have a $1$\dash separation.
Therefore we can let $v$ be a degree-two vertex.
Let $\gamma$ be a $\Gamma$\dash gaining of $\Omega/e$ where $e$ is one of the two edges incident to $v$. 
    But now the gain function obtained by extending $\gamma$ to $\Omega$ by labelling $e$ with $\identity{\Gamma}$ is a $\Gamma$\dash gaining of $\Omega$, a contradiction.

    Now $\Omega$ is biconnected with minimum degree at least three. 
    Finally suppose that $\Omega$ has a 2-separation $(X,Y)$ for which every cycle in $X$ is balanced and $X$ contains at least two edges.
    Let $u$ and $v$ be the two vertices in $V(X)\cap V(Y)$.
    Then $X$ contains a path from $u$ to $v$, or else $\Omega$ has a $1$\dash separation.
    Let $P$ be such a path and let $\Omega'$ be the biased graph obtained from $\Omega$ by contracting all but one edge (call it $e$) from $P$, and deleting every edge of $X$ that is not in $P$.
    A least one element has been deleted or contracted from $X$, so $\Omega'$ is gainable over $\Gamma$. 
    Let $\gamma'$ be a $\Gamma$\dash gaining of $\Omega'$; by switching if necessary we may assume $\gamma'$ assigns $\identity{\Gamma}$ to $e$. 
    Let $\gamma$ be the $\Gamma$\dash gain function on $\Omega$ whose restriction to $Y$ is equal to $\gamma'$ and maps each edge in $X$ to $\identity{\Gamma}$. 
    By balanced rerouting, every cycle $C$ meeting both $X$ and $Y$ has the same bias as the cycle $(C \cap Y) \cup P$.
    Therefore $C$ is balanced in $\Omega$ if and only if $(C\cap Y)\cup e$ is balanced in $\Omega'$.
    Every cycle contained in $X$ is balanced, so it follows that $\gamma$ is a $\Gamma$\dash gaining of $\Omega$, and we have a contradiction. 
\end{proof}

Let $e$ be a non-loop edge of the biased graph $\Omega$, and assume that every unbalanced cycle contains $e$.
In this case we say that $e$ is a \emph{balancing edge} of $\Omega$.
If there is at least one unbalanced cycle, $C$, then $C$ contains $e$.
Since any cycle not containing $e$ is balanced, it follows that any cycle containing $e$ can be obtained from $C$ by balanced rerouting.
Therefore any cycle containing $e$ is unbalanced.

\begin{proposition}\label{prop:balancing-edge}
Let $\Omega$ be a biased graph with a balancing edge.
If $\Gamma$ is a non-trivial group then $\Omega$ is $\Gamma$\dash gainable.
\end{proposition}

\begin{proof}
If $\Omega$ has no unbalanced cycle we can assign the identity of $\Gamma$ to every edge.
Therefore we assume that $C$ is an unbalanced cycle, which necessarily contains $e$.
As we explained, a cycle is unbalanced if and only if it contains $e$.
We define the gaining $\gamma$ so that $\gamma$ sends every edge other than $e$ to the identity, and so that $\gamma$ does not take $e$ to the identity.
Now $\gamma$ is a $\Gamma$\dash gaining of $\Omega$.
\end{proof}

Recall that if $G$ is a graph and $n$ is a positive integer, then $nG$ is the graph obtained from $G$ by replacing each non-loop edge by a parallel class of $n$ edges.

\begin{proposition}
\label{nK2-exminor}
Let $\Gamma$ be a non-trivial group and let $n$ be a positive integer.
If $n\leq |\Gamma|$ then $(nK_{2},\emptyset)$ is $\Gamma$\dash gainable, and if $n=|\Gamma|+1$, then $(nK_{2},\emptyset)$ is an excluded minor for $\Gamma$\dash gainability.
\end{proposition}

\begin{proof}
Let $1$ and $2$ be the two vertices in $nK_{2}$.
If $n\leq |\Gamma|$, then we can let $\gamma$ be a gaining such that whenever $e$ and $f$ are different edges, we have $\gamma(e,1,2)\ne \gamma(f,1,2)$.
This establishes the first statement.
Now let $n=|\Gamma|+1$.
Assume for a contradiction that $\gamma$ is a gaining of $(nK_{2},\emptyset)$.
By the pigeonhole principle, there exist distinct edges $e$ and $f$ such that $\gamma(e,1,2)=\gamma(f,1,2)$.
Now the cycle comprising $e$ and $f$ is balanced, which is a contradiction.
Deleting any edge from $(nK_{2},\emptyset)$ produces a $\Gamma$\dash gainable biased graph, by the first statement.
Contracting any edge produces a biased graph which is gainable over any non-trivial group.
\end{proof}

Recall that the biased graph $\pm G$ is obtained by replacing each non-loop edge of $G$ with a parallel class of two edges, one positive and one negative, where a cycle is balanced in $\pm G$ if and only if it contains an even number of negative edges.

\begin{proposition}
\label{pmK3-exminor}
Let $\Gamma$ be a non-trivial group.
If $\Gamma$ contains an element of order two, then $\pm K_{3}$ is $\Gamma$\dash gainable.
Otherwise $\pm K_{3}$ is an excluded minor for $\Gamma$\dash gainability.
\end{proposition}

\begin{proof}
Assume that the vertices of $\pm K_{3}$ are $1$, $2$, and $3$.
Let $a$ and $b$ be the parallel edges joining $1$ to $2$, let $c$ and $d$ be the edges joining $2$ and $3$, and let $e$ and $f$ be the edges joining $1$ and $3$.
We assume that a cycle is balanced if and only if it contains an even number of edges from $\{b,d,f\}$.
If $h$ is an element of $\Gamma$ with order two, then we can let $\gamma$ be the gaining such that
\[\gamma(a,1,2)=\gamma(c,2,3)=\gamma(e,3,1)=\identity{\Gamma}\]
and
\[
\gamma(b,1,2)=\gamma(d,2,3)=\gamma(f,3,1)=h.
\]
This proves the first statement.

Now let $\Gamma$ be a non-trivial group, and assume that $\gamma$ is a $\Gamma$\dash gaining of $\pm K_{3}$.
By applying a switching, we can assume that $\gamma$ sends the edges $a$ and $c$ to $\identity{\Gamma}$.
Since $\{a,c,e\}$ is a balanced cycle, it now follows that $\gamma$ sends $e$ to $\identity{\Gamma}$.
Let $\alpha$ be the element $\gamma(b,1,2)$.
Since $\{a,b\}$ is an unbalanced cycle, the order of $\alpha$ is greater than one.
Because $\{b,d,e\}$ is a balanced cycle, it follows that $\gamma(d,2,3)=\alpha^{-1}$.
Now $\{d,f,a\}$ is a balanced cycle, so $\gamma(f,3,1)=\alpha$.
As $\{b,f,c\}$ is balanced, this means that $\alpha^{2}=\identity{\Gamma}$.
So now we have established that $\pm K_{3}$ is $\Gamma$\dash gainable if and only if $\Gamma$ contains an element of order two.

We consider the proper minors of $\pm K_{3}$.
Let $\Gamma$ be an arbitrary non-trivial group and let $\alpha$ be a non-identity element.
First we suppose that we will delete an edge.
Since we can apply a switching so that the edges of an arbitrary spanning tree receive the identity element, we can assume without loss of generality that we are going to delete the edge $b$.
Let $\gamma$ be the $\Gamma$\dash gaining that takes $a$, $c$, and $e$ to the identity element and where $\gamma(d,2,3)=\gamma(f,1,3)=\alpha$.
It is not difficult to check that the only cycles in $\cB_{\gamma}$ have the edge sets $\{a,c,e\}$ and $\{a,d,f\}$, so $\gamma$ is a gaining of $\pm K_{3}\backslash b$, as desired.

Next we consider contracting the edge $a$.
Let $w$ be the vertex obtained via identifying $1$ and $2$.
Let $\gamma$ be the gaining where $c$ and $e$ are taken to $\identity{\Gamma}$ and where
\[\gamma(b,w,w)=\gamma(d,3,w)=\gamma(f,3,w)=\alpha.\]
The only cycles in $\cB_{\gamma}$ have edge sets $\{c,e\}$ or $\{d,f\}$, so $\gamma$ is a gaining of $\pm K_{3}/a$.
\end{proof}

Recall that the biased graph $-G$ consists of the graph $G$ along with the collection of all even-length cycles.

\begin{proposition}
\label{-K4-exminor}
Let $\Gamma$ be a non-trivial group.
If $\Gamma$ contains an element of order two, then $-K_{4}$ is $\Gamma$\dash gainable.
Otherwise $-K_{4}$ is an excluded minor for $\Gamma$\dash gainability.
\end{proposition}

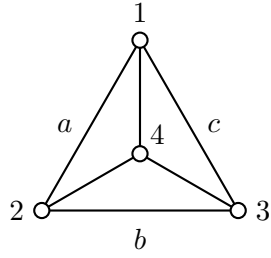
\begin{figure}
\centering

\begin{tikzpicture}

\node[name = u1] at (90:1.5cm) {};
\node[name = u2] at (210:1.5cm) {};
\node[name = u3] at (-30:1.5cm) {};
\node[name = u4] at (0:0cm) {};

\draw[thick] (u1.center) -- (u2.center) node [pos = 0.5, left = 1mm] {$a$};
\draw[thick] (u2.center) -- (u3.center) node [pos = 0.5, below = 1mm] {$b$};
\draw[thick] (u3.center) -- (u1.center) node [pos = 0.5, right = 1mm] {$c$};
\draw[thick] (u4.center) -- (u1.center);
\draw[thick] (u4.center) -- (u2.center);
\draw[thick] (u4.center) -- (u3.center);

\node[above = 1mm] at (u1) {$1$};
\node[left = 1mm] at (u2) {$2$};
\node[right = 1mm] at (u3) {$3$};
\node[above right] at (u4) {$4$};

\filldraw[thick, fill = white] (u1) circle (1mm);
\filldraw[thick, fill = white] (u2) circle (1mm);
\filldraw[thick, fill = white] (u3) circle (1mm);
\filldraw[thick, fill = white] (u4) circle (1mm);

\end{tikzpicture}

\caption{A labelling of the vertices and edges in $K_{4}$.}
\label{fig:-K4}
\end{figure}

\begin{proof}
Let the vertices of $-K_{4}$ be labelled as in \cref{fig:-K4}.
We label three of the edges $a$, $b$, and $c$.
If $h$ is an element of $\Gamma$ with order two, then we can let $\gamma$ be the gaining that takes every edge to $h$, and then a cycle is in $\cB_{\gamma}$ if and only if it has an even number of edges.
Therefore $\gamma$ is a $\Gamma$\dash gaining of $-K_{4}$.

Now let $\Gamma$ be a non-trivial group, and assume that $\gamma$ is a $\Gamma$\dash gaining of $-K_{4}$.
By applying a switching, we can assume that $\gamma$ sends every edge incident with $4$ to $\identity{\Gamma}$.
Let $\alpha$ be $\gamma(a,1,2)$, and note that $\alpha\ne \identity{\Gamma}$.
The cycle containing $a$ and $b$ but not $c$ is balanced, so $\gamma(b,2,3)=\alpha^{-1}$.
The cycle containing $b$ and $c$ but not $a$ is balanced, which means that $\gamma(c,3,1)=\alpha$.
Now the cycle containing $a$ and $c$ but not $b$ implies that $\alpha^{2}=\identity{\Gamma}$, so $\Gamma$ has an element of order two.

Let $\Gamma$ be an arbitrary non-trivial group and let $\alpha$ be a non-identity element.
The gaining $\gamma$, where the edges incident with $4$ are taken to $\identity{\Gamma}$ and where $\gamma(b,2,3)=\gamma(c,1,3)=\alpha$, is a gaining of $-K_{4}\backslash a$.
We can easily check that contracting any edge from $-K_{4}$ produces the same biased graph as deleting an edge from $\pm K_{3}$.
Therefore \cref{pmK3-exminor} implies that any such contraction of $-K_{4}$ is $\Gamma$\dash gainable.
This shows that any proper minor of $-K_{4}$ is gainable over any non-trivial group.
\end{proof}

\section{Uniqueness of gainings}

Establishing the uniqueness of $\mathbb{F}_{3}$\dash representations is a critical step in finding the excluded minors for the class of ternary matroids.
In this section we carry out the analogous step for $\ZZ_{3}$\dash gainable biased graphs by showing that $\ZZ_{3}$\dash gainings are essentially unique.
Neudauer and Slilaty show that for any finite group $\Gamma$, there is an upper bound on the number of inequivalent gainings of vertically $3$\dash connected biased graphs~\cite{MR3575200}.
So our results show that this upper bound is one, when $|\Gamma|=3$.

\begin{definition}[Rounded ordering]
\label{def:rounded}
Let $G$ be a connected graph and let $F$ be a set of edges such that $G[F]$ contains a spanning tree of $G$.
Let $S$ be the set of non-loop edges not in $F$.
Let $s_{0},s_{1},\ldots, s_{t}$ be an ordering of $S$ such that for every $i$ the following conditions hold:
\begin{enumerate}[label = \textup{(\roman*)}]
\item there is a block, $B_{i}$, of $G[F\cup \{s_{0},s_{1},\ldots, s_{i}\}]$ that contains all of the edges $s_{0},s_{1},\ldots, s_{i}$,
\item between any pair of distinct vertices in $B_{i}$ there is a path of $B_{i}$ that uses only edges in $F$ as well as a path that uses exactly one edge in $\{s_{0},s_{1},\ldots, s_{i}\}$, and
\item if $i>0$, then $s_{i}$ is in a cycle of $B_{i}$ that contains exactly one edge in $\{s_{0},s_{1},\ldots, s_{i-1}\}$.
\end{enumerate}
In this case we say that $s_{0},s_{1},\ldots, s_{t}$ is a \emph{rounded} ordering of $S$.
\end{definition}

\begin{proposition}
\label{prop:rounded}
Let $G$ be a connected graph with no proper $1$\dash separations.
Let $F$ be a set of edges such that $G[F]$ contains a spanning tree of $G$.
Let $S$ be the set of non-loop edges not in $F$.
There is a rounded ordering of $S$.
\end{proposition}

\begin{proof}
If $S$ is empty the result holds vacuously, so we assume $S\ne \emptyset$.
This implies that $G$ has at least two vertices.
Let $s_{0},s_{1},\ldots, s_{k}$ be any sequence of distinct elements of $S$.
We will say the sequence is \emph{good} if conditions (i), (ii), and (iii) from \cref{def:rounded} hold for every $i\in \{0,1,\ldots, k\}$.

We start by showing that there is a non-empty good sequence.
Let $s_{0}$ be an arbitrary edge in $S$.
Since $G[F]$ contains a spanning tree, it follows that $G[F\cup s_{0}]$ contains a cycle that contains $s_{0}$.
This cycle is contained in a block $B_{0}$ of $G[F\cup s_{0}]$.
Let $x$ and $y$ be an arbitrary pair of distinct vertices in $B_{0}$.
Since $B_{0}$ is biconnected, there is a cycle $C$ that contains $s_{0}$ and $x$.
Again using biconnectivity, we can find a $y$\dash $C$ path $P$ of $B_{0}$ such that $x$ is not in $P$.
Let $y'$ be the vertex of $P$ that is in $C$.
Thus $y'\ne x$.
There are two subpaths of $C$ that join $x$ and $y'$.
By concatenating these subpaths with $P$, we obtain a path of $B_{0}$ between $x$ and $y$ that uses only edges in $F$, and another path between $x$ and $y$ that contains exactly one edge in $S$.
This shows that $s_{0}$ is a non-empty good sequence.

Choose $s_{0},s_{1},\ldots, s_{k}$ to be a good sequence of maximum length.
Let $B_{k}$ be the block of $G[F\cup\{s_{0},s_{1},\ldots, s_{k}\}]$ that contains $\{s_{0},s_{1},\ldots, s_{k}\}$.
Note that $k\geq 0$ and that $B_{k}$ has at least two vertices.
If $S=\{s_{0},s_{1},\ldots, s_{k}\}$, then we have a rounded ordering of $S$ and we are done, so we assume that there is some edge in $S-\{s_{0},s_{1},\ldots, s_{k}\}$.
First assume there exists such an edge, $e$, joining two vertices $u$ and $v$ in $B_{k}$.
We set $s_{k+1}$ to be $e$.
We will show that $s_{0},s_{1},\ldots, s_{k+1}$ is a good sequence.
Because $s_{k+1}$ joins two vertices of $B_{k}$, there is a block of $G[F\cup\{s_{0},s_{1},\ldots, s_{k+1}\}]$ that contains $B_{k}$ and $s_{k+1}$.
Let this block be $B_{k+1}$.
To show that $s_{0},s_{1},\ldots, s_{k+1}$ is a good sequence, the only condition that requires checking is that there is a cycle of $B_{k+1}$ that contains $s_{k+1}$ and exactly one edge in $\{s_{0},s_{1},\ldots, s_{k}\}$.
This is true because $B_{k}$ has a path between $u$ and $v$ that contains exactly one edge in $\{s_{0},s_{1},\ldots, s_{k}\}$.
Now we have contradicted the maximality of $s_{0},s_{1},\ldots, s_{k}$.
Therefore we will now assume that there is no edge in $S-\{s_{0},s_{1},\ldots, s_{k}\}$ that joins two vertices in $B_{k}$.

Because $S-\{s_{0},s_{1},\ldots, s_{k}\}$ is non-empty and any edge in this set is incident with a vertex not in $B_{k}$, there is at least one vertex not in $B_{k}$.
This means that $G[F\cup\{s_{0},s_{1},\ldots, s_{k}\}]$ has at least two blocks, and therefore a cut-vertex.
Let $w_{1}$ be a cut-vertex of $G[F\cup\{s_{0},s_{1},\ldots, s_{k}\}]$ that is contained in $B_{k}$.
Let $H$ be the connected component of
\[
G[F\cup\{s_{0},s_{1},\ldots, s_{k}\}] - w_{1}
\]
that contains $B_{k}-w_{1}$.
Since $w_{1}$ is a cut-vertex, not every vertex of $G[F\cup\{s_{0},s_{1},\ldots, s_{k}\}]$ is contained in $H$.
As $G$ has no cut-vertex, $G-w_{1}$ is connected.
Therefore some edge of $G-w_{1}$ joins a vertex of $H$ to a vertex that is not in $H$.
Let $s_{k+1}$ be such an edge.
Note that $s_{k+1}$ is not in $F$, since it joins two different connected components of $G[F\cup\{s_{0},s_{1},\ldots, s_{k}\}] - w_{1}$.
Assume that $s_{k+1}$ joins the vertices $u$ and $v$, where $v$ is in $H$ and $u$ is not.
Since $G[F\cup\{s_{0},s_{1},\ldots, s_{k}\}]$ is connected, we can let $P_{1}$ be a $u$\dash $B_{k}$ path of this graph.
Note that $w_{1}$ must be in $P_{1}$, because otherwise $P_{1}$ is a path of $G[F\cup\{s_{0},s_{1},\ldots, s_{k}\}]-w_{1}$ that connects $u$ to a vertex of $H$.
Therefore $w_{1}$ is the unique vertex of $B_{k}$ that is in $P_{1}$.
We can also let $P_{2}$ be a $v$\dash $(B_{k}-w_{1})$ path of $G[F\cup\{s_{0},s_{1},\ldots, s_{k}\}]-w_{1}$.
Let $w_{2}$ be the unique vertex of $B_{k}$ in $P_{2}$.
Since $B_{k}$ contains all the edges $s_{0},s_{1},\ldots, s_{k}$, it follows that both $P_{1}$ and $P_{2}$ are $F$\dash paths.
Let $P$ be the concatenation of $P_{1}$, $s_{k+1}$, and $P_{2}$.
Thus $P$ is a path of $G[F\cup\{s_{0},s_{1},\ldots, s_{k+1}\}]$ from $w_{1}$ to $w_{2}$ and every edge of $P$ aside from $s_{k+1}$ is in $F$.

As $w_{1}$ and $w_{2}$ are distinct vertices of $B_{k}$, there must be some block of $G[F\cup\{s_{0},s_{1},\ldots, s_{k+1}\}]$ that contains $B_{k}$ and $P$.
Let $B_{k+1}$ be this block.
There is a path of $B_{k}$ from $w_{1}$ to $w_{2}$ that uses a single edge in $\{s_{0},s_{1},\ldots, s_{k}\}$.
The union of this path with $P$ is a cycle of $B_{k+1}$ that contains $s_{k+1}$ and a single edge in $\{s_{0},s_{1},\ldots, s_{k}\}$.
To show that $s_{0},s_{1},\ldots, s_{k+1}$ is a good sequence, it now suffices to show that between any pair of vertices in $B_{k+1}$, there is a path of $B_{k+1}$ that uses only edges in $F$ as well as a path that uses a single edge in $\{s_{0},s_{1},\ldots, s_{k+1}\}$.

Let $x$ and $y$ be an arbitrary pair of distinct vertices in $B_{k+1}$.
If both $x$ and $y$ are in $B_{k}$ then our claimed paths exist by the inductive hypothesis, so we assume that $x$ is not in $B_{k}$.
Let $P_{x}'$ be a minimal path of $B_{k+1}$ that contains $x$ and a vertex in the union of $B_{k}$ and $P$.
Note that any edge of $P_{x}'$ belongs to $F$, since $
s_{0},s_{1},\ldots, s_{k}$ are all contained in $B_{k}$ and $s_{k+1}$ is in $P$.
Let $x'$ be the unique vertex of $P_{x}'$ that is in either $B_{k}$ or $P$.
If $x'$ is in $P$, then by choosing an appropriate subpath of $P$, we can find a path of $G[F]$ from $x'$ to a vertex in $B_{k}$.
In any case, there is a path of $B_{k+1}$ from $x$ to a vertex of $B_{k}$ that uses only edges in $F$.
Let this path be $P_{x}$ and let $x_{1}$ be the vertex of $P_{x}$ that is in $B_{k}$.
If $y$ is in $B_{k}$, then there is a path of $B_{k}$ between $y$ and $x_{1}$ that uses only edges in $F$, as well as a path that uses a single edge in $\{s_{0},s_{1},\ldots, s_{k+1}\}$.
By concatenating these paths with $P_{x}$, we obtain our desired paths of $B_{k+1}$.
Therefore we assume that $y$ is not in $B_{k}$.

By the same argument as before, there is a path $P_{y}$ of $B_{k+1}$ which joins $y$ to a vertex in $B_{k}$ and uses only edges in $F$.
Let $y_{1}$ be the vertex in $P_{y}$ that is in $B_{k}$.
Now $x_{1}$ and $y_{1}$ are joined by a path of $B_{k}$ that is also a path of $G[F]$.
By concatenating this path with $P_{x}$ and $P_{y}$ we obtain a walk from $x$ to $y$ using edges in $F$.
Therefore $B_{k+1}$ contains a path of $G[F]$ from $x$ to $y$.
Now we need only show that there is a path of $B_{k+1}$ from $x$ to $y$ that uses a single edge in $\{s_{0},s_{1},\ldots, s_{k+1}\}$.

Since $B_{k+1}$ is biconnected, Menger's Theorem says there is a pair of disjoint paths joining $\{x,y\}$ to vertices in $B_{k}$.
Assume that both paths are minimal with this property, so that each one contains a unique vertex in $B_{k}$.
Assume that one of these paths joins $x$ to $x_{2}$ and the other joins $y$ to $y_{2}$.
If both of these paths are paths of $G[F]$, then we concatenate them with a path of $B_{k}$ that uses a single edge in $\{s_{0},s_{1},\ldots, s_{k}\}$ to obtain the path we desire.
Otherwise, $s_{k+1}$ is in exactly one of the paths, and the other is a path of $G[F]$.
Now we concatenate the pair of paths with a path of $B_{k}$ from $x_{2}$ to $y_{2}$ that uses only edges in $F$.
We obtain a path of $B_{k+1}$ that contains $s_{k+1}$ and no edges in $\{s_{0},s_{1},\ldots, s_{k}\}$ so we are done.
We have shown that $s_{0},s_{1},\ldots, s_{k+1}$ is a good sequence, contradicting the choice of $s_{0},s_{1},\ldots, s_{k}$.
\end{proof}

\begin{proposition}
\label{prop:two-edges-in-cycle}
Let $\gamma$ and $\gamma'$ be $\ZZ_{3}$\dash gainings of the biased graph $\Omega$. 
Let $e_{1}$ and $e_{2}$ be distinct edges in a cycle $C$, and let $\arc{a}_{1}$ and $\arc{a}_{2}$ be arcs of $e_{1}$ and $e_{2}$ respectively.
Assume that $\gamma$ and $\gamma'$ take every edge in $C$ other than $e_{1}$ and $e_{2}$ to $1$ and that neither $\gamma$ nor $\gamma'$ takes $\arc{a}_{1}$ or $\arc{a}_{2}$ to $1$.
If $\gamma(\arc{a}_{1})=\gamma'(\arc{a}_{1})$ then $\gamma(\arc{a}_{2})=\gamma'(\arc{a}_{2})$.
\end{proposition}

\begin{proof}
Without loss of generality we can assume that $W$ is a simple closed walk contained in $C$ such that $\arc{a}_{1}$ and $\arc{a}_{2}$ are both arcs of $W$ and that $\arc{a}_{1}$ appears in $W$ before $\arc{a}_{2}$.
Assume that $\gamma(\arc{a}_{1})=\gamma'(\arc{a}_{1})$ and write $\pi$ for this common value.
The gain of $W$ is equal to $\pi\gamma(\arc{a}_{2})$ under $\gamma$ and equal to $\pi\gamma'(\arc{a}_{2})$ under $\gamma'$.
If $C$ is balanced, then
\[
\pi\gamma(\arc{a}_{2})=1=\pi\gamma'(\arc{a}_{2})
\]
which implies the result.
Now we assume that $C$ is not balanced.
This means that $\pi\gamma(\arc{a}_{2}) \ne 1$.
Since the two terms in this product are non-identity elements of $\ZZ_{3}$, we conclude that $\pi = \gamma(\arc{a}_{2})$.
The same argument shows $\pi = \gamma'(\arc{a}_{2})$ so the result follows.
\end{proof}

\begin{proposition}
\label{prop:equal-on-spanning-tree}
Let $\Omega$ be a connected biased graph and let $\gamma$ and $\gamma'$ be $\ZZ_{3}$\dash gainings of $\Omega$.
Let $T$ be a spanning tree of $\Omega$ and assume that $\gamma$ and $\gamma'$ both send all edges of $T$ to $1$.
\begin{enumerate}[label = \textup{(\roman*)}]
\item If every non-loop cycle is balanced, then $\gamma$ and $\gamma'$ both send all non-loop edges to $1$.
\item If $\Omega$ has no proper $1$\dash separation, then there exists an automorphism $\tau$ of $\ZZ_{3}$ such that $\gamma'(\arc{a})=\tau(\gamma(\arc{a}))$ whenever \arc{a} is an arc of a non-loop edge.
\end{enumerate}
\end{proposition}

\begin{proof}
\cref{prop:agree-on-tree} shows that $\gamma$ and $\gamma'$ take the same edges to $1$.
Let $F$ be the set of non-loop edges that are taken to $1$ by $\gamma$ and $\gamma'$.
Therefore $F$ contains all the edges of $T$.
Let $S$ be the set of non-loop edges that are not in $F$.

Assume that every non-loop cycle is balanced.
Let \arc{a} be an arc of an arbitrary non-loop edge, $e$.
Let $C$ be the cycle contained in the union of $T$ and $e$.
Since $C$ is balanced we see that $\gamma(\arc{a})=1=\gamma'(\arc{a})$, so part (i) follows.

Now we assume that $\Omega$ has no proper $1$\dash separation.
If $S$ is empty then every non-loop cycle is balanced and we have reduced to case (i).
Therefore we may assume that $S$ is non-empty.
\cref{prop:rounded} shows that there is a rounded ordering $s_{0},s_{1},\ldots, s_{t}$ of $S$.
For each $i$ let $\arc{a}_{i}$ be an arc of $s_{i}$.
Let $B_{i}$ be the block of $\Omega[F\cup\{s_{0},s_{1},\ldots, s_{i}\}]$ that contains $\{s_{0},s_{1},\ldots, s_{i}\}$.

If $\gamma'(\arc{a}_{0})=\gamma(\arc{a}_{0})$ then we let $\tau$ be the identity automorphism.
Otherwise, we let $\tau$ be the non-trivial automorphism of $\ZZ_{3}$.
We can now rename $\tau\circ \gamma'$ as $\gamma'$ and therefore assume $\gamma'(\arc{a}_{0})=\gamma(\arc{a}_{0})$.
We need to show that $\gamma'(\arc{a}_{i})=\gamma(\arc{a}_{i})$ for every $i$.
This equation holds when $i=0$, so we make the inductive assumption that it holds when $i\leq k$.
Let $C$ be a cycle of $B_{k+1}$ that contains $s_{k+1}$ and exactly one edge $s_{j}$ in $\{s_{0},s_{1},\ldots, s_{k}\}$.
The inductive assumption implies $\gamma'(\arc{a}_{j})=\gamma(\arc{a}_{j})$.
Now \cref{prop:two-edges-in-cycle} implies $\gamma'(\arc{a}_{k+1})=\gamma(\arc{a}_{k+1})$, so the inductive step is complete.
\end{proof}

\begin{corollary}
\label{prop:unique-gaining}
Let $\Omega$ be a connected biased graph such that either every non-loop cycle is balanced, or $\Omega$ has no proper $1$\dash separation.
Assume that $\gamma$ and $\gamma'$ are $\ZZ_{3}$\dash gainings of $\Omega$.
By performing a switching on $\gamma'$ we can obtain a $\ZZ_{3}$\dash gaining $\rho$ such that there exists an automorphism $\tau$ of $\ZZ_{3}$ with the property that $\rho(\arc{a})=\tau(\gamma(\arc{a}))$ whenever \arc{a} is an arc of a non-loop edge.
\end{corollary}

\begin{proof}
Let $T$ be a spanning tree of $\Omega$.
Let $\eta\colon V(\Omega)\to \ZZ_{3}$ and $\nu\colon V(\Omega)\to \ZZ_{3}$ be switching functions such that the gainings $\gamma^{\eta}$ and $(\gamma')^{\nu}$ both send every edge in $T$ to $1$.
By \cref{prop:equal-on-spanning-tree}, there is an automorphism $\tau$ such that $(\gamma')^{\nu}(\arc{a})=\tau(\gamma^{\eta}(\arc{a}))$ whenever \arc{a} is an arc of a non-loop edge.
Now let $\mu\colon V(\Omega)\to\ZZ_{3}$ be the function taking each $v\in V(\Omega)$ to $\tau(\eta(v)^{-1})$.
We let $\rho$ be $((\gamma')^{\nu})^{\mu}$.
Whenever $\arc{a}=(e,u,v)$ is an arc of a non-loop edge we have
\begin{align*}
\rho(\arc{a})&=((\gamma')^{\nu})^{\mu}(\arc{a})\\
&=\mu(u)^{-1}(\gamma')^{\nu}(\arc{a})\mu(v)\\
&=\tau(\eta(u))\tau(\gamma^{\eta}(\arc{a}))\tau(\eta(v)^{-1})\\
&=\tau(\eta(u))\tau(\eta(u)^{-1}\gamma(\arc{a})\eta(v))\tau(\eta(v)^{-1})\\
&=\tau(\eta(u)\eta(u)^{-1})\tau(\gamma(\arc{a}))\tau(\eta(v)\eta(v)^{-1})\\
&=\tau(\gamma(\arc{a})).\qedhere
\end{align*}
\end{proof}

Our next lemma mirrors the strategy used by Gerards in a short proof characterising the excluded minors for regular matroids~\cite{MR0986875}.
A similar strategy appears in the proof characterising the excluded minors for $\mathbb{F}_{4}$\dash representable matroids~\cite{MR1769191}.
That proof involves building an $\mathbb{F}_{4}$\dash representable ``twin" of the excluded minor.
The next lemma builds a biased graph that acts as a $\ZZ_{3}$\dash gainable twin of an excluded minor.

\begin{lemma}
\label{lem:twinlemma}
Let $e$ and $f$ be distinct non-loop edges in the biased graph $\Omega=(G,\cB)$ such that:
\begin{enumerate}[label = \textup{(\roman*)}]
\item $\Omega\backslash e$ and $\Omega\backslash f$ are both $\ZZ_{3}$\dash gainable,
\item $\Omega\backslash \{e,f\}$ is connected up to isolated vertices and has no proper $1$\dash separation,
\item $\Omega\backslash \{e,f\}$ contains a non-loop unbalanced cycle,
\item if $e$ is pendant in $\Omega\backslash f$, then it is pendant in $\Omega$, and
\item if $f$ is pendant in $\Omega\backslash e$, then it is pendant in $\Omega$.
\end{enumerate}
There exists a unique set of cycles $\cB_{\textup{twin}}$ satisfying the theta property such that $\Omega_{\textup{twin}}=(G,\cB_{\textup{twin}})$ is $\ZZ_{3}$\dash gainable and
\begin{equation}
\label{eqn1}
\Omega_{\textup{twin}}\backslash e = \Omega\backslash e
\qquad\text{and}\qquad
\Omega_{\textup{twin}}\backslash f = \Omega\backslash f.
\end{equation}
\end{lemma}

\begin{proof}
There is a connected component of $\Omega\backslash \{e,f\}$ that contains a non-loop unbalanced cycle.
(Every other component is an isolated vertex.)
Let $T'$ be a spanning tree of this connected component, and let $T$ be a maximal forest of $\Omega$ such that $T'$ is a subtree of $T$.
Thus any edge of $T$ that is not in $T'$ is either $e$ or $f$.
Let $\gamma_{e}$ be a $\ZZ_{3}$\dash gaining of $\Omega\backslash e$ and let $\gamma_{f}$ be a  $\ZZ_{3}$\dash gaining of $\Omega\backslash f$.
By applying a switching, we can assume that $\gamma_{e}$ sends every edge in $T\backslash e$ to $1$ and that $\gamma_{f}$ sends every edge in $T\backslash f$ to $1$.
Therefore $\gamma_{e}$ and $\gamma_{f}$ both send the edges of $T'$ to $1$.
\cref{prop:equal-on-spanning-tree} tells us that after applying an automorphism of $\ZZ_{3}$ if necessary, we can assume $\gamma_{e}(\arc{a})=\gamma_{f}(\arc{a})$ whenever \arc{a} is an arc of a non-loop edge of $\Omega\backslash \{e,f\}$.

Now we construct $\gamma$, a $\ZZ_{3}$\dash gaining of $G$.
Let \arc{a} be an arc of $G$.
If \arc{a} is the arc of a loop, then we set $\gamma(\arc{a})$ to be $1$ if that loop is in $\cB$, and set it to be $\omega$ otherwise.
Now we assume that \arc{a} is an arc of a non-loop edge.
If \arc{a} is an arc of an edge in $G\del \{e,f\}$, then $\gamma(\arc{a})$ is defined to be $\gamma_{e}(\arc{a}) = \gamma_{f}(\arc{a})$.
On the other hand, if \arc{a} is an arc of $e$, then we set 
$\gamma(\arc{a})$ to be $\gamma_{f}(\arc{a})$ and if \arc{a} is an arc of $f$, we set $\gamma(\arc{a})$ to be $\gamma_{e}(\arc{a})$.
This completes the construction of $\gamma$.
We let $\cB_{\textup{twin}}$ be $\cB_{\gamma}$.
It is clear that~\eqref{eqn1} holds.

Next we must prove uniqueness.
Let $\cB'$ be a set of cycles satisfying the theta property such that $\Omega'=(G,\cB')$ is $\ZZ_{3}$\dash gainable and~\eqref{eqn1} holds.
A loop edge belongs to $\cB$ if and only if it belongs to $\cB_{\textup{twin}}$, by~\eqref{eqn1}.
Let $\gamma'$ be a $\ZZ_{3}$\dash gaining of $\Omega'$.
By switching we can assume that $\gamma'$ sends every edge in $T$ to $1$.
If we restrict $\gamma$ and $\gamma'$ to the edges of $\Omega\backslash \{e,f\}$, we obtain two gainings of this graph that send every edge in the spanning tree $T'$ to $1$.
\cref{prop:equal-on-spanning-tree} implies that by applying an automorphism of $\ZZ_{3}$ as necessary, we can assume that $\gamma(\arc{a})=\gamma'(\arc{a})$ whenever \arc{a} is an arc of a non-loop edge in $\Omega\backslash \{e,f\}$.
We will be done if we can show that $\gamma(\arc{a}) = \gamma'(\arc{a})$ whenever \arc{a} is an arc of a non-loop edge in $\Omega$, and this is true if \arc{a} is an arc of a non-loop edge other than $e$ or $f$.
So we will assume that \arc{a} is an arc of $e$ (the case when it is an arc of $f$ is identical).

Since $\gamma$ and $\gamma'$ send the edges of $T'$ to $1$, \cref{prop:agree-on-tree} implies that these gainings send exactly the same edges of $\Omega\backslash \{e,f\}$ to $1$.
Let $F$ be the set of edges of $\Omega\backslash \{e,f\}$ that are taken to $1$ by $\gamma$ and $\gamma'$.
Let $S$ be the set of non-loop edges in $\Omega\backslash \{e,f\}$ that are not in $F$.
We have assumed that $\Omega\backslash \{e,f\}$ contains a non-loop unbalanced cycle.
This means that there must be at least one edge in $S$, so condition (iii) in \cref{def:rounded} applies.
\cref{prop:rounded} implies that there is a rounded ordering of $S$.
Consequently, between any pair of (non-isolated) vertices in $\Omega\backslash \{e,f\}$ there is a path of $\Omega\backslash \{e,f\}$ that uses only edges in $F$, along with a path that uses exactly one edge in $S$.

Recall that \arc{a} is an arc of $e$.
We first assume that $e$ joins two vertices of $T'$.
Let $P$ be a path of $\Omega\backslash\{e,f\}$ between the vertices of $e$ such that $P$ uses exactly one edge in $S$ and otherwise uses edges in $F$.
Let $P'$ be a path of $\Omega\backslash \{e,f\}$ between the same pair of vertices where $P'$ uses only edges in $F$.
Assume that $\gamma$ takes $e$ to $1$.
Then the union of $e$ with $P'$ is a balanced cycle of $\Omega\backslash f = \Omega'\backslash f$.
It follows that $\gamma'$ also sends $e$ to $1$.
The same argument shows that if $\gamma'$ sends $e$ to $1$, then $\gamma$ does also.
Therefore we may as well assume that neither $\gamma$ nor $\gamma'$ sends $e$ to $1$.
The union of $P$ with $e$ is a cycle of $\Omega\backslash f=\Omega'\backslash f$, and it contains exactly two edges that are not sent to $1$ by $\gamma$ or $\gamma'$.
Now it follows from \cref{prop:two-edges-in-cycle} that $\gamma(\arc{a})=\gamma'(\arc{a})$.
So if $e$ joins two vertices of $T'$, then $\gamma(\arc{a})=\gamma'(\arc{a})$ in any case.

Next assume that $e$ does not join two vertices of $T'$.
We let $w$ be a vertex incident with $e$ such that $w$ is not in $T'$.
Note that no edge of $\Omega\backslash \{e,f\}$ is incident with $w$, so $e$ is a pendant edge in $\Omega\backslash f$.
But then the hypotheses imply that $e$ is a pendant edge in $\Omega$, so $e$ is the only edge of $\Omega$ that is incident with $w$.
As $T$ is a maximal forest of $\Omega$, it follows that $e$ is in $T$.
Since $\gamma_{f}$ sends edges of $T\backslash f$ to $1$, we see that both $\gamma$ and $\gamma'$ send $e$ to $1$.
This concludes the argument that $\gamma(\arc{a})=\gamma'(\arc{a})$ whenever \arc{a} is an arc of $e$.
Exactly the same argument applies when \arc{a} is an arc of $f$.
Therefore $\cB'=\cB_{\gamma'}=\cB_{\gamma}=\cB_{\textup{twin}}$ holds and the proof is complete.
\end{proof}

\section{The case with a balancing pair of edges}

Let $e$ be a non-loop edge in the biased graph $\Omega$.
Recall that if every unbalanced cycle of $\Omega$ contains $e$, then $e$ is a balancing edge.
Now let $e$ and $f$ be distinct non-loop edges.
If every unbalanced cycle contains either $e$ or $f$, then $\{e,f\}$ is a \emph{balancing pair} of edges.
Assume that $\{e,f\}$ is a balancing pair and that $f$ is not a balancing edge.
Then there is an unbalanced cycle that contains $e$ but not $f$.
Since $e$ is a balancing edge of $\Omega\del f$, we can use balanced rerouting to show that any cycle containing $e$ but not $f$ is unbalanced.
A symmetrical statement applies if $e$ is not a balancing edge.

Let $\{e,f\}$ be a balancing pair of edges in the biased graph $\Omega$.
Let the end-vertices of $e$ be $u_{e}$ and $v_{e}$ and let the end-vertices of $f$ be $u_{f}$ and $v_{f}$.
Let $C$ and $D$ be two cycles that contain $e$ and $f$.
Assume that there are simple closed walks $W_{C}$ and $W_{D}$, contained in $C$ and $D$ respectively, such that $W_{C}$ contains $u_{e}, v_{e},  u_{f}, v_{f}$ in this order and $W_{D}$ contains the sequence $u_{e}, v_{e},  v_{f}, e_{f}$.
In this case we say that $C$ and $D$ have \emph{opposite} directions.
Otherwise $C$ and $D$ have the \emph{same} direction.
Let $C$ and $D$ be distinct cycles containing $e$ and $f$.
We say that $C, D$ is an \emph{inconsistent pair} of cycles if either:
\begin{enumerate}[label = $\bullet$]
\item $C$ and $D$ have the same direction but opposite bias, or,
\item $C$ and $D$ have opposite directions but the same bias.
\end{enumerate}

\begin{lemma}\label{inconsistent-cycles}
Let $\Omega$ be a biased graph containing an unbalanced cycle.
Let $\{e,f\}$ be a balancing pair of edges in  
$\Omega$, where neither $e$ nor $f$ is a balancing edge.
Then $\Omega$ is $\ZZ_{3}$\dash gainable if and only if it does not contain a pair of inconsistent cycles.
\end{lemma}

\begin{proof}
Assume that the end-vertices of $e$ (respectively $f$) are $u_{e}$ and $v_{e}$ (respectively $u_{f}$ and $v_{f})$.
Let $\arc{a}_{e}$ be $(e,u_{e},v_{e})$ and let $\arc{a}_{f}$ be $(f,u_{f},v_{f})$.
Suppose $\gamma$ is a $\ZZ_3$\dash gaining of $\Omega$.
Since $\{e,f\}$ is a balancing pair of edges, by switching we can assume that $\gamma$ takes every edge other than $e$ and $f$ to $1$.
As neither $e$ nor $f$ is a balancing edge, $\gamma$ takes neither $e$ nor $f$ to $1$.
Assume for a contradiction that $\Omega$ contains a pair of inconsistent cycles, $C$ and $D$.
First assume that $C$ and $D$ have the same direction.
Without loss of generality we can assume that $W_{C}$ and $W_{D}$ are simple closed walks contained in $C$ and $D$ respectively such that $\arc{a}_{e}$ and $\arc{a}_{f}$ are arcs of both $W_{C}$ and $W_{D}$ and that $\arc{a}_{e}$ appears in these walks before $\arc{a}_{f}$.
Then
\[
\gamma(W_{C}) = \gamma(\arc{a}_{e})\gamma(\arc{a}_{f}) = \gamma(W_{D})
\]
and therefore both $C$ and $D$ are balanced or both are unbalanced.
This contradicts the definition of an inconsistent pair, so $C$ and $D$ have opposite directions.
We can assume that $W_{C}$ and $W_{D}$ are simple closed walks contained in $C$ and $D$ such that $W_{C}$ contains $\arc{a}_{e}$ and $\arc{a}_{f}$ in this order, while $W_{D}$ contains $\arc{a}_{e}$ and $\arc{a}_{f}^{-1}$ in this order.
Then $\gamma(W_C) = \gamma(\arc{a}_{e})\gamma(\arc{a}_{f})$ and
\[
\gamma(W_{D}) = \gamma(\arc{a}_{e})\gamma(\arc{a}_{f}^{-1}) = \gamma(\arc{a}_{e})\gamma(\arc{a}_{f})^{-1}.
\]
Since $\gamma(\arc{a}_{e})$ and $\gamma(\arc{a}_{f})$ are non-identity elements of $\ZZ_{3}$, it is impossible for $\gamma(\arc{a}_{e})\gamma(\arc{a}_{f})$ and $\gamma(\arc{a}_{e})\gamma(\arc{a}_{f})^{-1}$ to be simultaneously equal to $1$, or not equal to $1$.
Therefore $C$ and $D$ have opposite bias and we have a contradiction.
This shows that if $\Omega$ is $\ZZ_{3}$\dash gainable, there is no pair of inconsistent cycles.

Conversely, suppose $\Omega$ does not contain a pair of inconsistent cycles.
Because neither $e$ nor $f$ is a balancing edge, there is an unbalanced cycle containing $e$ but not $f$, and another containing $f$ but not $e$.
We noted earlier that any cycle containing $e$ but not $f$, or $f$ but not $e$, must be unbalanced.
Assume there is no cycle containing both $e$ and $f$.
In this case we can define the gaining $\gamma$ so that $\gamma(\arc{a}_{e})=\gamma(\arc{a}_{f})=\omega$, and $\gamma$ takes any other edge to $1$.
It is easy to see that $\gamma$ is a $\ZZ_{3}$\dash gaining of $\Omega$.
Therefore we consider the case that there is at least one cycle, $C$, containing both $e$ and $f$.
We can assume that any simple closed walk in $C$ contains both $\arc{a}_{e}$ and $\arc{a}_{f}$ or both $\arc{a}_{e}^{-1}$ and $\arc{a}_{f}^{-1}$.
Let $D$ be an arbitrary cycle containing $e$ and $f$.
Because there is no pair of inconsistent cycles, if $C$ and $D$ have the same direction then they have same bias, and if they have opposite directions, then they have opposite bias.
Assume $C$ is balanced.
We define the gaining $\gamma$ so that $\gamma(\arc{a}_{e})=\omega$ and $\gamma(\arc{a}_{f})=\omega^{2}$.
On the other hand, if $C$ is unbalanced we define $\gamma$ so that $\gamma(\arc{a}_{e})=\gamma(\arc{a}_{f})=\omega$.
In either case, define $\gamma$ so that it takes any other edge to $1$.
Now it is not hard to see that $\gamma$ is a $\ZZ_{3}$\dash gaining of $\Omega$.
\end{proof}

Assume $\{e,f\}$ is a balancing pair of edges of the biased graph $\Omega$, and that $e$ and $f$ are incident with a common vertex.
Then all cycles that contain both $e$ and $f$ have the same direction, meaning that no inconsistent pair of cycles can exist.

\begin{lemma}\label{balancingpair}
Let $\Omega$ be an excluded minor for the class of $\ZZ_3$\dash gainable graphs.  
If $\Omega$ has a balancing pair of edges, then $\Omega$ is isomorphic to $-K_{4}$.
\end{lemma}

\begin{proof}
Certainly $\Omega$ has an unbalanced cycle, or else it is gainable over every group.
Because $\Omega$ is not $\ZZ_{3}$\dash gainable, \cref{prop:balancing-edge} implies that it has no balancing edge.
\cref{exminors2conn} implies that $\Omega$ is biconnected (and therefore loopless), and has no $2$\dash separation with a balanced side that contains two edges.
Furthermore, $\Omega$ has minimum degree at least three.

Let $\{e, f\}$ be a balancing pair of edges in $\Omega$.
Let $u_{e}$ and $v_{e}$ be the end-vertices of $e$ and let $u_{f}$ and $v_{f}$ be the end-vertices of $f$.
Call a cycle consisting of $\{e,f\}$, a $v_{e}$\dash $u_{f}$ path, and a $v_{f}$\dash $u_{e}$ path a \emph{forward cycle}.
Forward cycles traverse the end-vertices of $e$ and $f$ in the cyclic order $u_{e},v_{e},u_{f},v_{f}$. 
Let us call a cycle that is the union of $\{e,f\}$, a $v_{e}$\dash $v_{f}$ path, and a $u_{f}$\dash $u_{e}$ path a \emph{backward cycle}.
Backward cycles traverse the end-vertices of $\{e,f\}$ in the cyclic order $u_{e},v_{e},v_{f},u_{f}$.
So a pair of cycles has opposite directions if and only if one is a forward cycle and the other is backward.

Since $\Omega$ is not $\ZZ_{3}$\dash gainable, \cref{inconsistent-cycles} implies that it contains a pair of inconsistent cycles.
Therefore $e$ and $f$ are not incident with a common vertex, by the paragraph preceding this lemma.

\begin{claim} \label{Whinnie-P}
Let $C$ and $D$ be a pair of inconsistent cycles.
Every edge of $\Omega$ is in either $C$ or $D$.
\end{claim}

\begin{proof}
Assume for a contradiction that $\Omega$ contains an edge $a$ not in $C \cup D$.
Then $\{e,f\}$ is a balancing pair in $\Omega\del a$ and $C$ and $D$ form an inconsistent pair of cycles in $\Omega\del a$.
If neither $e$ nor $f$ is a balancing edge in $\Omega \del a$, then \cref{inconsistent-cycles} implies that $\Omega\del a$ is not $\ZZ_3$\dash gainable, contradicting the minimality of $\Omega$.
Therefore we will assume without loss of generality that $e$ is a balancing edge of $\Omega\del a$, so that any unbalanced cycle of $\Omega\del a$ contains $e$.
Note that $f$ is a balancing edge of $\Omega\del e$, and therefore a cycle of $\Omega\del e$ is unbalanced if and only if it contains $f$.
But $a$ is also a balancing edge of $\Omega\del e$, since $\Omega\del \{a,e\}$ has no unbalanced cycles.
Therefore a cycle of $\Omega\del e$ is unbalanced if and only if it contains $a$.
This shows that $a$ and $f$ are in exactly the same cycles in $\Omega\del e$.
Furthermore, $\Omega\del e$ has at least one unbalanced cycle, so neither $a$ nor $f$ is an isthmus of $\Omega\del e$.
From this discussion we see that $\{a,f\}$ is a minimal edge cut-set in $\Omega\del e$.

There must be exactly two connected components in $\Omega\del \{e,f,a\}$.
Let these components be $H_{1}$ and $H_{2}$.
Assume that $e$ joins two vertices of $H_{1}$.
Then no cycle of $H_{2}$ contains $e$ or $f$, and therefore any such cycle is balanced in $\Omega$.
If $H_{2}$ contains two edges, then $\Omega$ has a $2$\dash separation where one side is balanced and contains at least two edges.
If it contains a single edge, then this edge is a non-loop, so $\Omega$ contains a vertex of degree at most two.
If $H_{2}$ contains no edge at all, then once again $\Omega$ has a vertex of degree at most two.
In any case we have a contradiction.
A symmetrical argument shows that $e$ cannot join two vertices of $H_{2}$, so $e$ joins a vertex of $H_{1}$ to a vertex of $H_{2}$, and therefore $\{e,f,a\}$ is a minimal edge cut-set of $\Omega$.

By changing the labels on $u_{e}$, $v_{e}$, $u_{f}$, and $v_{f}$ as necessary, we can assume that $H_{1}$ contains $v_{e}$ and $u_{f}$ and that $H_{2}$ contains $u_{e}$ and $v_{f}$.
Now any backwards cycle has to contain the edge $a$ twice, and hence no such cycle exists.
So $C$ and $D$ are both forwards cycles, and because they form an inconsistent pair, they have opposite bias.
Let $P_{C}$ and $P_{D}$, respectively, be the subpaths of $C$ and $D$ from $v_{e}$ to $u_{f}$ that are contained in $H_{1}$.
Note that $(C-P_{C})\cup P_{D}$ is a cycle of $\Omega$ because $P_{D}$ is internally disjoint with $C$.
As $H_{1}$ is balanced, it follows that $(C-P_{C})\cup P_{D}$ is obtained from $C$ by balanced rerouting and therefore $(C-P_{C})\cup P_{D}$ has the same bias as $C$.
Let $Q_{C}$ and $Q_{D}$ be the subpaths of $C$ and $D$ from $v_{f}$ to $u_{e}$ that are contained in $H_{2}$.
Now we can obtain $D$ from $(C-P_{C})\cup P_{D}$ by replacing $Q_{C}$ with $Q_{D}$.
The same reasoning tells us that $D$ has the same bias as $(C-P_{C})\cup P_{D}$, and hence the same bias as $C$.
Therefore we have a contradiction.
\end{proof}

A \emph{chord} of a cycle is a non-loop edge that joins two vertices of the cycle without being contained in it.

\begin{claim}\label{Galinda}
Let $C$ and $D$ be a pair of inconsistent cycles and let $a$ be an edge not equal to $e$ or $f$.
Then $a$ is a chord of either $C$ or $D$.
\end{claim}

\begin{proof}
Note that \cref{Whinnie-P} implies that $a$ is an edge of either $C$ or $D$.
Assume that $a$ is a chord of neither $C$ nor $D$.
If $a$ is in $C$ then we let $C'$ be $C/a$, and otherwise we let $C'$ be $C$.
We define $D'$ in the same way.
Therefore $C'$ and $D'$ are both cycles in $\Omega/a$ that contain $e$ and $f$, and indeed $C'$ and $D'$ form an inconsistent pair of cycles in $\Omega/a$.

Now $\{e,f\}$ is a balancing pair of edges in $\Omega/a$.
If neither $e$ nor $f$ is a balancing edge in $\Omega/a$, then \cref{balancingpair} implies that $\Omega/a$ is not $\ZZ_{3}$\dash gainable and we have a contradiction.
Therefore we will assume without loss of generality that $e$ is a balancing edge in $\Omega/a$.
This means that $\Omega\del e/a$ is balanced.
But $\Omega\del e$ has an unbalanced cycle, and $a$ is not a loop, so we have contradicted \cref{prop:unbalanced-contract}.
\end{proof}

\begin{claim} \label{Simba}
Let $C$ and $D$ be a pair of inconsistent cycles.
Both $C$ and $D$ are Hamiltonian.
An edge that is in both $C$ and $D$ is equal to either $e$ or $f$.
\end{claim}

\begin{proof}
Suppose to the contrary, without loss of generality, that $z$ is a vertex not in $D$.
Let $a$ be an edge incident to $z$.
Since $a$ is not in $D$, it is in $C$ by \cref{Whinnie-P}.
So $a$ is not a chord of $C$, and it cannot be a chord of $D$ since $z$ is not in $D$.
Thus we have contradicted \cref{Galinda}.
Therefore both $C$ and $D$ are Hamiltonian.
Now assume $a$ is an edge in both $C$ and $D$, so that it is a chord of neither cycle.
\cref{Galinda} implies that $a$ is equal to $e$ or $f$.
\end{proof} 

\begin{claim}\label{Adams}
Let $C$ and $D$ be a pair of inconsistent cycles.
Let $P$ and $Q$ be the two maximal subpaths of $C$ that contain neither $e$ nor $f$.
Every edge in $D$ has one end-vertex in $P$ and the other in $Q$.
\end{claim}

\begin{proof}
Suppose not, and let $d$ be an edge in $D$ each of whose end-vertices are in $P$.
This means $d$ is not equal to either $e$ or $f$.
The cycle in $P \cup d$ is balanced, so via a balanced rerouting of $C$ along $d$ we obtain a cycle $C'$ containing $d$ with the same bias as $C$.
Furthermore $C$ and $C'$ have the same direction.
Therefore $C'$ and $D$ form an inconsistent pair of cycles that have the edge $d$ in common.
This contradicts \cref{Simba}.
\end{proof}

\begin{claim}\label{clm1}
There exists an inconsistent pair of cycles with opposite direction.
\end{claim}

\begin{proof}
We assume this fails.
Since $\Omega$ is not $\ZZ_{3}$\dash gainable, \cref{inconsistent-cycles} tells us that there is an inconsistent pair, $C$ and $D$.
By changing the labels of $u_{f}$ and $v_{f}$ as necessary, we can assume that $C$, and therefore $D$, is a forward cycle.
Note that $C$ and $D$ have opposite bias.

Let $C_{1}$ be the subpath of $C$ that joins $v_{e}$ to $u_{f}$ while avoiding $e$ and $f$, and let $C_{2}$ be the subpath of $C$ joining $v_{f}$ to $u_{e}$ that avoids $e$ and $f$.
Let $d_{u}$ and $d_{v}$ (respectively) be the vertices that are joined to $u_{e}$ (respectively $v_{e}$) by edges of $D$.
Now $u_{e}$ is in $C_{2}$, so \cref{Adams} implies $d_{u}$ is in $C_{1}$.
Similarly, $d_{v}$ is in $C_{2}$.
Let $D'$ be the cycle consisting of $e$, $f$, the edges of $D$ that join $u_{e}$ to $d_{u}$ and $v_{e}$ to $d_{v}$, along with the subpaths $d_{u}C_{1}u_{f}$ and $d_{v}C_{2}v_{f}$.
We observe that $D'$ is a backward cycle.

If $D'$ has the same bias as $C$ then $C$ and $D'$ form an inconsistent pair with opposite direction, contrary to assumption.
So $D'$ has the opposite bias to $C$, meaning that $D'$ has the same bias as $D$.
Therefore $D$ and $D'$ form an inconsistent pair with opposite directions, and we have a contradiction.
\end{proof}

At this point, we assume that $C$ and $D$ are an inconsistent pair of cycles with opposite directions.
As in the proof of \cref{clm1}, we can assume that $C$ is a forward cycle, so that $D$ is a backward cycle.
Note that $C$ and $D$ have the same bias.
We define $D'$ in exactly the same way as in the proof of \cref{clm1}, so that $D'$ is a backward cycle.
We claim that $C$ and $D'$ form an inconsistent pair.
This is true if $C$ and $D'$ have the same bias, so we assume that they have opposite bias.
Thus $D$ and $D'$ have the opposite bias and the same direction.
We see that $D$ and $D'$ form an inconsistent pair.
But they share two edges other than $e$ and $f$, so we have contradicted \cref{Simba}.
This contradiction shows that $C$ and $D'$ form an inconsistent pair.

If either $d_{u}C_{1}u_{f}$ or $d_{v}C_{2}v_{f}$ contains an edge, then $C$ and $D'$ have a common edge that is not $e$ or $f$, contradicting \cref{Simba}.
Therefore $d_{u}C_{1}u_{f}$ and $d_{v}C_{2}v_{f}$ are trivial paths, so $d_{u}=u_{f}$ and $d_{v}=v_{f}$.
This shows that $D$ contains only four vertices.
As $D$ is Hamiltonian by \cref{Simba}, we see that $\Omega$ contains exactly four vertices.
Every edge is of $\Omega$ is in either $C$ or $D$ by \cref{Whinnie-P}.
Now \cref{Simba} implies that $\Omega$ has exactly six edges.
No edge of $C$ is parallel to an edge of $D$, so $\Omega$ has no parallel pairs of edges and no loops, and therefore the underlying graph of $\Omega$ is $K_{4}$.
The edges $e$ and $f$ form a matching.

Any $3$\dash edge cycle of $\Omega$ contains exactly one of $e$ and $f$, and is therefore unbalanced.
Assume that $C$ is unbalanced.
Then $D$ too is unbalanced because they have the same bias.
By contracting $e$ and $f$ we obtain the biased graph $(4K_{2},\emptyset)$, which is an excluded minor for $\ZZ_{3}$\dash gainability by \cref{nK2-exminor}.
Thus $\Omega$ has a proper minor that is not $\ZZ_{3}$\dash gainable, and we have a contradiction.
We conclude that $C$, and therefore $D$, is balanced.

There is one more $4$\dash edge cycle of $\Omega$: the cycle consisting of all edges other than $e$ and $f$.
Denote this cycle by $A$ and assume that $A$ is unbalanced.
We will show that $\Omega$ is $\ZZ_{3}$\dash gainable.
Let $x$ be the edge joining $v_{e}$ to $u_{f}$ and let $y$ be the edge joining $v_{f}$ to $u_{e}$.
Let $\gamma$ be the gaining that sends every edge of $D$ to $1$, where $\gamma(x,v_{e},u_{f})=\omega$, and where $\gamma(y,v_{f},u_{e})=\omega^{2}$.
It is straightforward to check that this is indeed a $\ZZ_{3}$\dash gaining of $\Omega$, which is a contradiction.
Therefore $A$ is balanced.

We have shown that the underlying graph of $\Omega$ is $K_{4}$, that every $3$\dash edge cycle is unbalanced, and every $4$\dash edge cycle is balanced.
Thus $\Omega$ is $-K_{4}$, as desired.
\end{proof}

\section{Proof of the main theorem}

In this section we prove our main theorem.

\begin{theorem}
\label{thm:main-theorem}
The only excluded minors for the class of $\ZZ_{3}$\dash gainable biased graphs are $(4K_{2}, \emptyset)$, $\pm K_{3}$, and $-K_{4}$.
\end{theorem}

\begin{proof}
Let $\Omega = (G,\cB)$ be a biased graph that is an excluded minor for the class of $\ZZ_3$\dash gainable biased graphs.
Assume for a contradiction that $\Omega$ is not one of the biased graphs listed in the theorem statement.
Since the excluded minors form an antichain in the minor order, this means that $\Omega$ has no minor isomorphic to $(4K_{2},\emptyset)$, $\pm K_{3}$, or $-K_{4}$.

It follows from \cref{prop:balancing-edge} that $\Omega$ has no balancing edge.
Since $\Omega$ has no minor isomorphic to $-K_{4}$, \cref{balancingpair} implies that $\Omega$ has no balancing pair of edges.
\Cref{exminors2conn} says that $\Omega$ has at least two vertices, minimum degree at least three, and is biconnected.
Recall that this means that $\Omega$ has no $1$\dash separation.
In particular, $\Omega$ is loopless.
Furthermore, any $2$\dash edge cycle is unbalanced.

\begin{claim}\label{clm:at-least-3-vertices}
$\Omega$ has at least three vertices.
\end{claim}

\begin{proof}
Assume that $\Omega$ has exactly two vertices.
Since $\Omega$ is loopless and every $2$\dash edge cycle is unbalanced, $\Omega$ is isomorphic to $(nK_{2},\emptyset)$ for some positive integer $n$.
\cref{nK2-exminor} now implies that either $\Omega$ is $\ZZ_{3}$\dash gainable, or it contains a $(4K_{2},\emptyset)$\dash minor.
We have a contradiction in either case, so $\Omega$ has at least three vertices.
\end{proof}

Since $\Omega$ is biconnected with at least three vertices, it is $2$\dash connected.
\cref{helper} says that we can choose distinct edges $e$ and $f$ such that each of $\Omega \backslash e$, $\Omega \backslash f$, and $\Omega \backslash \{e,f\}$ is $2$\dash connected.
Because $\Omega$ has no balancing pair, it follows that $\Omega \backslash \{e,f\}$ has an unbalanced cycle, and this cycle is not a loop, since $\Omega$ is loopless.
Now we can apply \cref{lem:twinlemma} to obtain the unique biased graph $\Omega_{\text{twin}} = (G,\cB_{\text{twin}})$ such that $\Omega_{\text{twin}}$ is $\ZZ_3$\dash gainable and $\Omega_{\text{twin}} \backslash e = \Omega \backslash e$ and $\Omega_{\text{twin}} \backslash f = \Omega \backslash f$.
Since $\Omega$ is not $\ZZ_3$\dash gainable, but $\Omega_{\text{twin}}$ is, there must exist at least one cycle that is in exactly one of $\cB$ and $\cB_{\text{twin}}$.
Any such cycle is said to be a \emph{bad cycle} of $\Omega$.
As $\Omega_{\text{twin}} \backslash e = \Omega \backslash e$, it follows that any bad cycle contains $e$, and by symmetry, also contains $f$.

\begin{claim}
\label{clm:bad-cycle-in-minor}
Let $C$ be a bad cycle.
Let $\Lambda = \Omega/A\backslash B$ be a proper minor of $\Omega$ such that:
\begin{enumerate}[label = \textup{(\roman*)}]
\item $e$ and $f$ are edges of $\Lambda$,
\item $B$ does not contain any edge of $C$,
\item if $P$ is a path of $\Omega[A]$ that joins two distinct vertices of $C$, then every edge of $P$ is in $C$,
\item if $e$ is pendant in $\Lambda\backslash f$, then it is pendant in $\Lambda$, and
\item if $f$ is pendant in $\Lambda\backslash e$, then it is pendant in $\Lambda$.
\end{enumerate}
Then either $\Lambda\backslash\{e,f\}$ is not connected (even after removing isolated vertices), or it has a proper $1$\dash separation, or the only unbalanced cycles in $\Lambda\backslash\{e,f\}$ are loops.
\end{claim}

\begin{proof}
Assume for a contradiction that $\Lambda\backslash \{e,f\}$ is connected (up to isolated vertices), has no proper $1$\dash separation, and also contains a non-loop unbalanced cycle.
Certainly $\Lambda$ is $\ZZ_{3}$\dash gainable because it is a proper minor of $\Omega$, and therefore so are $\Lambda\backslash e$ and $\Lambda\backslash f$.
We may apply \cref{lem:twinlemma} to $\Lambda$ and find the unique $\ZZ_{3}$\dash gainable biased graph $\Lambda_{\text{twin}}$ on the same underlying graph as $\Lambda$ where $\Lambda_{\text{twin}}\backslash e = \Lambda\backslash e$ and $\Lambda_{\text{twin}}\backslash f = \Lambda\backslash f$.
But $\Lambda$ is also a $\ZZ_{3}$\dash gainable biased graph.
Furthermore, $\Lambda\backslash e = \Lambda\backslash e$ and $\Lambda\backslash f = \Lambda\backslash f$ trivially hold.
So the uniqueness of $\Lambda_{\text{twin}}$ now implies that $\Lambda=\Lambda_{\text{twin}}$.

Because $\Omega_{\text{twin}}\backslash e = \Omega \backslash e$, we see that
\[
(\Omega_{\text{twin}}/A\backslash B)\backslash e
= (\Omega_{\text{twin}}\backslash e)/A\backslash B
=(\Omega\backslash e)/A\backslash B
=(\Omega / A \backslash B)\backslash e = \Lambda\backslash e.
\]
Similarly, $(\Omega_{\text{twin}}/A\backslash B)\backslash f=\Lambda\backslash f$.
Now $\Omega_{\text{twin}}$ is $\ZZ_{3}$\dash gainable, and therefore so is $\Omega_{\text{twin}}/A\backslash B$.
Uniqueness implies that $\Omega_{\text{twin}}/A\backslash B = \Lambda_{\text{twin}}=\Lambda = \Omega/A\backslash B$.

The hypotheses imply that $\Lambda$ contains a cycle $C'$ obtained from $C$ by contracting every edge of $C$ that is in $A$.
Because $C$ is a bad cycle, this means that $C'$ is balanced in exactly one of $\Lambda = \Omega/A\backslash B$ and $\Omega_{\text{twin}}/A \backslash B$.
This is a contradiction to $\Omega_{\text{twin}}/A\backslash B = \Omega/A\backslash B$, so the proof is complete.
\end{proof}

Next we show that any bad cycle must be Hamiltonian.
Recall that a chord of a cycle is a non-loop edge that joins two vertices of the cycle without being contained in it.
A chord is \emph{proper} if it is not parallel to an edge of the cycle.

\begin{claim}\label{clm:bad-Hamiltonian}
Any bad cycle contains every vertex of $\Omega$.
\end{claim}

\begin{proof}
Let $C$ be a bad cycle, so that $C$ contains $e$ and $f$.
Suppose there exists a vertex $z \in V(\Omega) \setminus V(C)$.
Note that $z$ is not incident with $e$ or $f$.
Recall that there is a non-loop unbalanced cycle in $\Omega \backslash \{e,f\}$.
Since any unbalanced cycle contains at most two edges incident with $z$, but $z$ has degree at least three, we can choose an edge $g$ incident with $z$ such that $\Omega \backslash \{e,f,g\}$ has a non-loop unbalanced cycle.
Assume that every unbalanced cycle of $\Omega\backslash \{e,f\}/g$ is a loop.
We will prove that $g$ is in an unbalanced $2$\dash edge cycle of $\Omega\del \{e,f\}$.
Let $A$ be an unbalanced cycle of $\Omega\del \{e,f\}$.
If $g$ is in $A$, then $A$ must be a $2$\dash edge cycle, as we want, because otherwise $\Omega\backslash \{e,f\}/g$ has an unbalanced non-loop cycle.
So we will assume that $g$ is not in $A$.
Then $g$ is a chord of $A$, so $A\cup g$ contains two cycles that contain $g$.
At least one of these cycles is unbalanced, or else we violate the theta property.
If $g$ is a proper chord of $A$, then there is an unbalanced  cycle in $A\cup g$ that contains $g$ and at least three edges.
This contradicts our assumption, so $g$ is not a proper chord of $A$, and therefore it is in a $2$\dash edge cycle, which must be unbalanced.

We have proved that if every unbalanced cycle of $\Omega\backslash \{e,f\}/g$ is a loop, then $g$ is in an unbalanced $2$\dash edge cycle of $\Omega\del \{e,f\}$.
In this case, we let $h$ be an edge incident with $z$ that is not parallel to $g$.
(Note that such an edge exists because otherwise $\Omega\backslash \{e,f\}$ would have a proper $1$\dash separation.)
Since there is an unbalanced cycle consisting of two edges parallel with $g$, we see that both $\Omega\backslash \{e,f,h\}$ and $\Omega\backslash \{e,f\}/h$ have non-loop unbalanced cycles.
Now we relabel $h$ as $g$, so in any case we can assume that $g$ is an edge incident with $z$ chosen so that both $\Omega\backslash \{e,f,g\}$ and $\Omega\backslash \{e,f\}/g$ have non-loop unbalanced cycles.

Since $\Omega\backslash \{e,f\}$ is $2$\dash connected and loopless, it follows from \cref{prop:delete-or-contract} that we can choose a graph from $\Omega \backslash \{e,f,g\}$ or $\Omega \backslash \{e,f\} / g$ that is connected without proper $1$\dash separations.
Let $\Lambda$ be either $\Omega/g$ or $\Omega \backslash g$, chosen so that $\Lambda\backslash \{e,f\}$ is connected and has no proper $1$\dash separation.
Note that $e$ and $f$ are edges of $\Lambda$ and that $g$ has been chosen so that it does not join two vertices of the bad cycle $C$ (since it is incident with $z$).
Assume that $e$ is pendant in $\Lambda\backslash f$.
Then $\Lambda$ must be $\Omega\backslash g$, and some degree-three vertex of $\Omega$ is incident with $e$, $f$, and $g$.
But in this case $g$ is a pendant edge of $\Omega\backslash \{e,f\}$, which is a contradiction since this graph is $2$\dash connected.
Therefore $e$ is not pendant in $\Lambda\backslash f$ and symmetrically, $f$ is not pendant in $\Lambda \backslash e$.
Now we apply \cref{clm:bad-cycle-in-minor} with $A$ set to $\{g\}$ and $B$ set to $\emptyset$.
Since $\Lambda\backslash \{e,f\}$ is connected and has no $1$\dash separation, we can deduce that its only unbalanced cycles are loops.
This is impossible because $\Lambda\backslash \{e,f\}$ is equal to either $\Omega\backslash \{e,f,g\}$ or $\Omega\backslash \{e,f\}/g$, and both these biased graphs have non-loop unbalanced cycles.
This contradiction shows that $C$ is a Hamiltonian cycle of $\Omega$.
\end{proof}

Let $C$ be a bad cycle.
Then $C$ contains $e$ and $f$, and is Hamiltonian by \cref{clm:bad-Hamiltonian}.
Since $\Omega$ has at least three vertices, we can deduce that $e$ and $f$ are not parallel.

\begin{claim}\label{clm:bad-containing-x}
Let $C$ be a bad cycle and let $g$ be an edge of $C$ that is not equal to $e$ or $f$.
Assume $\Omega\del \{e,f\}/g$ is connected and has no proper $1$\dash separation.
The only unbalanced cycles of $\Omega\del\{e,f\}$ that contain $g$ are $2$\dash edge cycles.
If $D$ is an unbalanced cycle of $\Omega\del\{e,f\}$ that does not contain $g$, then one of the edges of $D$ is parallel to $g$.
Moreover, $g$ is in the unique non-trivial parallel class of $\Omega\del\{e,f\}$.
\end{claim}

\begin{proof}
We apply \cref{clm:bad-cycle-in-minor} with $\Lambda = \Omega/g$.
Since $g$ is an edge of the bad cycle $C$, it is clear that conditions (i), (ii), and (iii) in this claim are true.
If $e$ is pendant in $\Lambda\del f$, then $\Lambda$ has a vertex with degree at most two, which implies that $\Omega$ too has such a vertex.
This is a contradiction, so by symmetry we see that conditions (iv) and (v) in \cref{clm:bad-cycle-in-minor} both hold.
Therefore the claim implies that $\Lambda\del \{e,f\}$ is not connected, or it has a proper $1$\dash separation, or its only unbalanced cycles are loops.
The hypotheses now imply that the only unbalanced cycles in $\Lambda\del \{e,f\}$ are loops.
Now it follows that the only unbalanced cycles of $\Omega\del \{e,f\}$ that contain $g$ are $2$\dash edge cycles.

Assume that $D$ is an unbalanced cycle of $\Omega\del \{e,f\}$ that does not contain $g$.
Note that $D$ is not a loop.
Then $g$ is a chord of $D$, or else $D$ is a non-loop unbalanced cycle of $\Omega\del \{e,f\}/g$.
At least one of the two cycles in $D\cup g$ that contains $g$ is unbalanced, by the theta property.
This cycle must be a $2$\dash edge cycle, so $g$ is parallel to an edge of $D$.

There is at least one unbalanced cycle in $\Omega\del \{e,f\}$.
The previous paragraphs now show that $g$ is in a $2$\dash edge cycle in $\Omega\del\{e,f\}$, so $g$ is in a non-trivial parallel class.
If $\Omega\del \{e,f\}$ contains a different non-trivial parallel class then it contains an unbalanced $2$\dash edge cycle that does not contain $g$ and which does not contain an edge parallel to $g$, which contradicts our previous conclusions.
\end{proof}

\begin{claim}\label{clm:four-vertices}
$\Omega$ has at least four vertices
\end{claim}

\begin{proof}
We proved in \cref{clm:at-least-3-vertices} that $\Omega$ has at least three vertices, so assume that it has exactly three.
Let these vertices be $u_{0}$, $u_{1}$, and $u_{2}$, where we can assume that $e$ joins $u_{0}$ and $u_{1}$ while $f$ joins $u_{0}$ and $u_{2}$.
Let $g$ be the third edge in the bad cycle $C$, so that $g$ joins $u_{1}$ and $u_{2}$.
If $e$ is not in a $2$\dash edge cycle, then $\Omega\del \{e,f\}$ cannot be $2$\dash connected, a contradiction.
Therefore we can assume that the edge $e'$ is parallel to $e$.
We can similarly assume that $f'$ is parallel to $f$.
Then $\Omega\del \{e,f\}/g$ is certainly connected, and it cannot have a proper $1$\dash separation because it only has two vertices.
So we can apply \cref{clm:bad-containing-x} and deduce that $g$ is in the unique non-trivial parallel class of $\Omega\del \{e,f\}$.
Let $g'$ be an edge that is parallel to $g$.
We see that $\{e,e'\}$ and $\{f,f'\}$ are both parallel classes of size two.

Now $\Omega\del e$, $\Omega\del g$, and $\Omega\del \{e,g\}$ are all $2$\dash connected.
Therefore we could construct $\Omega_{\text{twin}}$ relative to the pair $\{e,g\}$ instead of $\{e,f\}$.
In this case we would find a bad cycle that contains $e$, $g$, and either $f$ or $f'$.
By using the same arguments as above, we could show that $\{f,f'\}$ is the unique non-trivial parallel class of $\Omega\del \{e,g\}$.
But this implies that $\{g,g'\}$ is a parallel class of $\Omega$.
So now we know that $\Omega$ has three parallel classes, each of size two.

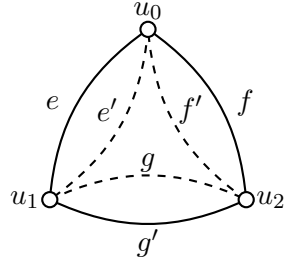
\begin{figure}[htb]
\centering

\begin{tikzpicture}

\node[name = u0] at (90:1.5cm) {};
\node[name = u1] at (210:1.5cm) {};
\node[name = u2] at (-30:1.5cm) {};

\draw[thick] (u0.center) to [bend right = 25] node [pos = 0.5, left = 1mm] {$e$} (u1.center);
\draw[thick] (u1.center) edge[bend right = 25] node [pos = 0.5, below = -0.8mm] {$g'$} (u2.center);
\draw[thick] (u2.center) edge[bend right = 25] node [pos = 0.5, right = 1mm] {$f$} (u0.center);
\draw[thick, dashed] (u0.center) to [bend left = 25] node [pos = 0.4, left = 0mm] {$e'$} (u1.center);
\draw[thick, dashed] (u1.center) edge[bend left = 25] node [pos = 0.5, above = -1mm] {$g$} (u2.center);
\draw[thick, dashed] (u2.center) edge[bend left = 25] node [pos = 0.6, right = 0mm] {$f'$} (u0.center);

\node[above] at (u0) {$u_{0}$};
\node[left] at (u1) {$u_{1}$};
\node[right] at (u2) {$u_{2}$};

\filldraw[thick, fill = white] (u0) circle (1mm);
\filldraw[thick, fill = white] (u1) circle (1mm);
\filldraw[thick, fill = white] (u2) circle (1mm);

\end{tikzpicture}

\caption{The biased graph $\Omega$. The solid lines are the edges in the set $X$.}
\label{fig:placeholder}
\end{figure}

Let $X = \{e,f,g'\}$.
We will show that a Hamiltonian cycle is balanced if and only if it contains an even number of edges in $X$.
This will show that $\Omega$ is isomorphic to $\pm K_{3}$, a contradiction.
The cycle which contains zero edges of $X$ is $\{e',f',g\}$, and 
\cref{clm:bad-containing-x} implies that this cycle is balanced, because it is a cycle of $\Omega\del \{e,f\}$ that contains $g$.
Assume $D$ and $D'$ are Hamiltonian cycles of $\Omega$ that have exactly two edges in common.
If one of these cycles is balanced, then the other is unbalanced, because otherwise the union of $D$ and $D'$ violates the theta property (remembering that every $2$\dash edge cycle is unbalanced).
Now it follows that any cycle with exactly one edge in $X$ is unbalanced.

Assume that $X$ is a balanced cycle.
Then any cycle that contains exactly two edges of $X$ is unbalanced, by the same argument as in the previous paragraph.
In this case a cycle is balanced if and only if it is $X$ or the complement of $X$.
We define the gaining $\gamma$ so that it every edge not in $X$ to $1$, and so that
\[
\gamma(e,u_{0},u_{1}) = \gamma(g',u_{1},u_{2}) = \gamma(f,u_{2},u_{0}) = \omega.
\]
Now we can check that $\gamma$ is a $\ZZ_{3}$\dash gaining of $\Omega$ and we have a contradiction.
Therefore $X$ is unbalanced.

If every cycle that contains exactly two edges of $X$ is balanced, then $\Omega$ is isomorphic to $\pm K_{3}$, and we have our desired contradiction.
So we assume there is at least one unbalanced cycle that contains exactly two edges of $X$.
Assume there is exactly one such cycle.
Using symmetry, we will assume that $\{e,f,g\}$ is unbalanced, but $\{e,f',g'\}$ and $\{e',f,g'\}$ are both balanced.
In this case we can define $\gamma$ so that once again it takes every edge not in $X$ to $1$ and so that
\[
\gamma(e,u_{0},u_{1}) = \gamma(f,u_{2},u_{0})=\omega
\quad\text{and}\quad
\gamma(g',u_{1},u_{2})=\omega^{2}.
\]
This again shows that $\Omega$ is $\ZZ_{3}$\dash gainable, so we have another contradiction.
Now we must assume that there are at least two unbalanced cycles that contain exactly two edges in $X$.
Without loss of generality, we can assume that $g'$ is in two such cycles.
Therefore $\{e,f',g'\}$ and $\{e',f,g'\}$ are both unbalanced.
We also know that $\{e,e'\}$, $\{f,f'\}$, $\{e',f',g'\}$, and $\{e,f,g'\} = X$ are unbalanced.
This means that $\Omega/g'\del g$ is isomorphic to $(4K_{2},\emptyset)$, which is a contradiction.
\end{proof}

\begin{claim}\label{clm:proper-chord-unbalanced}
Let $C$ be a bad cycle.
Let $a$ be a proper chord of $C$.
Let $C'$ be a cycle contained in $C\cup a$ that contains $a$.
Then $C'$ is unbalanced in both $\Omega$ and $\Omega_{\textup{twin}}$.
\end{claim}

\begin{proof}
Let $C_{1}$ and $C_{2}$ be the cycles contained in $C\cup a$ such that $a$ is in both $C_{1}$ and $C_{2}$.
Since $a$ is a proper chord, neither $C_{1}$ nor $C_{2}$ is Hamiltonian, so \cref{clm:bad-Hamiltonian} implies that neither is a bad cycle.
So $C_{1}$ is balanced in both $\Omega$ and $\Omega_{\textup{twin}}$, or unbalanced in both.
The same statement applies to $C_{2}$.
Recall that $C$ is balanced in exactly one of $\Omega$ and $\Omega_{\textup{twin}}$.
Let $\{\Omega_{\textup{bal}}, \Omega_{\textup{unbal}}\}$ be equal to $\{\Omega, \Omega_{\textup{twin}}\}$, where $C$ is balanced in $\Omega_{\textup{bal}}$ and unbalanced in $\Omega_{\textup{unbal}}$.
If $C_{1}$ and $C_{2}$ are both balanced, then in $\Omega_{\textup{unbal}}$, the theta subgraph consisting of $C$ and $a$ contains exactly two balanced cycles.
If exactly one of $C_{1}$ and $C_{2}$ is balanced, then we obtain the same contradiction to the theta property in $\Omega_{\textup{bal}}$.
So it must be the case that $C_{1}$ and $C_{2}$ are both unbalanced in both $\Omega$ and $\Omega_{\textup{twin}}$.
\end{proof}

We now fix the bad cycle $C$, so that $C$ contains $e$ and $f$.
Let $P$ and $Q$ be the two maximal subpaths of $C$ that contain neither $e$ nor $f$.
Because $C$ is Hamiltonian, every vertex of $\Omega$ is contained in either $P$ or $Q$.
Recall that $\Omega\del \{e,f\}$ is $2$\dash connected.
If $P$ and $Q$ both have at least two vertices, then Menger's Theorem implies that we can find disjoint edges $x$ and $y$ of $\Omega\del\{e,f\}$, both of which join a vertex of $P$ to a vertex of $Q$.
In this case we let $D$ be the cycle consisting of $x$ and $y$ along with the subpaths of $P$ and $Q$ that join the end-vertices of $x$ and $y$.
Now assume that $P$ contains a single vertex, call it $u$.
In $\Omega\del \{e,f\}$, there are two distinct neighbours of $u$ in $Q$, or else the biased graph is not $2$\dash connected.
So we let $D$ be the cycle consisting of the subpath of $Q$ that joins the two neighbours, along with edges joining them to $u$.
Similarly, if $Q$ contains a single vertex, then we can let $D$ be a cycle consisting of a non-trivial subpath of $P$ and two edges joining this subpath to the vertex in $Q$.
In any case, $D$ is a cycle of $\Omega\del\{e,f\}$ with at least three edges, containing exactly two edges joining vertices in $P$ to vertices in $Q$.
If both $P$ and $Q$ contain at least two vertices, then these two edges are disjoint.

\begin{claim}\label{clm:new-Hamiltonian}
$D$ is Hamiltonian.
\end{claim}

\begin{proof}
Note that $D$ contains at least three vertices.

\begin{subclaim}\label{subclm7}
If $g$ is an edge of $C-\{e,f\}$ such that $g$ is not in $D$ and $g$ shares a vertex with either $e$ or $f$, then $\Omega\del \{e,f\}/g$ is $2$\dash connected.
\end{subclaim}

\begin{proof}
Without loss of generality, we can assume that $g$ is in $P$ and that $g$ shares a vertex (call it $u$) with $e$.
Assume $\Omega\del \{e,f\}/g$ is not $2$\dash connected.
Because it has at least three vertices by \cref{clm:four-vertices}, it follows that $\Omega\del \{e,f\}/g$ has a proper $1$\dash separation, and hence a cut-vertex.
In fact, because $\Omega\del \{e,f\}$ is $2$\dash connected, the only cut-vertex of $\Omega\del \{e,f\}/g$ is obtained by identifying the end-vertices of $g$.
Let $w$ be this vertex.
Note that $D$ is a cycle of $\Omega\del \{e,f\}/g$, so some block contains $D$.
Let $v$ be a vertex of $(\Omega\del \{e,f\}/g)-w$ that is not in the same connected component as $D-w$.
Thus $v$ is not in $D$, not equal to $w$, and every $v$\dash $D$ path of $\Omega\del \{e,f\}/g$ contains the cut-vertex $w$.
But $v$ is in either $Q$ or $P-w$.
In the former case, there is a subpath of $Q$ that is a $v$\dash $D$ path of $(\Omega\del \{e,f\}/g)-w$, a contradiction.
So $v$ is in $P-w$.
If $v$ is in the subpath of $P$ from $w$ to $D$, then since $v \ne w$, there is a subpath of $P-w$ that joins $v$ to $D$ in $(\Omega\del \{e,f\}/g)-w$.
So therefore $v$ is not in this subpath.
Notice that $P$ contains more than two vertices, and therefore the intersection of $D$ and $P$ contains at least one edge.
Let $d_{1}$ and $d_{2}$ be the end-vertices of the maximal subpath of $P$ contained in $D$.
Thus $d_{1}\ne d_{2}$.
Without loss of generality, the minimal subpath of $P/g$ in $\Omega\del \{e,f\}/g$ from $w$ to a vertex in $D$ has $d_{1}$ as an end-vertex.
This means that $d_{2}\ne w$.
Therefore the minimal subpath of $P/g$ from $v$ to a vertex in $D$ contains $d_{2}$, and does not contain $w$.
Thus $v$ is joined to a vertex of $D$ by a path in $(\Omega\del \{e,f\}/g)-w$ and we have a final contradiction.
\end{proof}

Assume that $D$ is not Hamiltonian.
Then we can assume without loss of generality that the end-vertex of $P$ incident with $e$ is not in $D$.
Let $u$ be this end-vertex and let $g$ be the edge of $P$ incident with $u$.
Then $\Omega\del \{e,f\}/g$ is $2$\dash connected by \ref{subclm7}, so we are able to apply \cref{clm:bad-containing-x}.
The only unbalanced cycles of $\Omega\del \{e,f\}$ that contain $g$ are $2$\dash edge cycles.
Moreover, $g$ is in the unique non-trivial parallel class of $\Omega\del \{e,f\}$.
If an end-vertex of $Q$ were not in $D$, we could apply \ref{subclm7} again, and deduce that the only non-trivial parallel class of $\Omega\del \{e,f\}$ is incident with one of these end-vertices.
This would mean that the unique non-trivial parallel class of $\Omega\del \{e,f\}$ contains both an edge of $P$ and an edge of $Q$.
This is impossible, so both the end-vertices of $Q$ are in $D$.
Similarly, if we let $u'$ be the end-vertex of $P$ that is incident with $f$, then $u'$ must be in $D$.

Assume that $Q$ contains an edge and let $q$ be an edge in $Q$.
If $\Omega\del \{e,f\}/q$ is $2$\dash connected then we can apply \cref{clm:bad-containing-x} and deduce that $q$ is in the unique non-trivial parallel class of $\Omega\del \{e,f\}$, and we have the same contradiction as in the last paragraph.
Therefore $\Omega\del \{e,f\}/q$ is not $2$\dash connected.
Let $w$ be the vertex obtained by identifying the end-vertices of $q$.
Then $w$ is a cut-vertex of $\Omega\del \{e,f\}/q$.
Note that $D/q$ is a cycle of $\Omega\del \{e,f\}/q$, so there is some vertex $v$ not in the same component of $(\Omega\del \{e,f\}/q)-w$ as $(D/q)-w$.
Therefore any $v$\dash $(D/q)-w$ path in $\Omega\del \{e,f\}/q$ passes through $w$.
But $Q/q$ is a subpath of $D/q$, so all vertices of $Q/q$ are in $D/q$.
Therefore $v$ is in $P$.
But then a subpath of $P$ joins $v$ to $D/q$ in $(\Omega\del \{e,f\}/q)-w$, and we have a contradiction.
This shows that $Q$ cannot contain an edge and hence $Q$ consists of a single vertex.

Now $D$ has at least three edges.
It does not contain an edge of $Q$, because no such edge exists.
Therefore $D$ contains an edge $p$ that is in $P$.
Note that $p$ and $g$ are distinct edges of $P$.
If $\Omega\del \{e,f\}/p$ were $2$\dash connected, then $p$ would be in the unique non-trivial parallel class of $\Omega\del \{e,f\}$, which is impossible because we can say the same thing of $g$.
Therefore $\Omega\del \{e,f\}/p$ is not $2$\dash connected.
Let $w$ be the vertex produced by identifying the end-vertices of $p$.
Then $w$ is a cut-vertex of $\Omega\del \{e,f\}/p$, so there is some vertex $v$ such that any path from $v$ to the cycle $D/p$ passes through $w$.
Note that $v$ must be in $P$, but not in $D$

Recall that $u$ is the end-vertex of $e$ that is in $P$, and that $u$ is not in $D$.
Let $z$ be the other end-vertex of $P$.
We have shown that $z$ is in $D$.
Let $P'$ be the maximal subpath of $P$ in $\Omega\del \{e,f\}$ that contains no edge of $D$, so that $v$ is a vertex of $P'$.
One end-vertex of $P'$ is $u$.
Let us denote the other end-vertex by $u'$, so $u'$ is the unique vertex of $P'$ in $D$.
The subpath of $P'$ joining $v$ to $D$ must contain $w$ in $\Omega\del \{e,f\}/g$, which means that $u'$ is one of the two vertices identified as $w$.
Thus $p$ is the edge of $D$ that is in $P$ and incident with $u'$.
Since $\Omega\del \{e,f\}$ is $2$\dash connected, it has a $v$\dash $D$ path that does not contain $u'$.
Recall that such a path contains exactly one vertex of $D$.
This path must contain the other end-vertex of $p$, or else it is a $v$\dash $(D/p)$ path of $\Omega\del \{e,f\}/p$ that does not contain $w$.
Note that this path cannot contain the single vertex of $Q$, or else there would be another  $v$\dash $(D/p)$ path of $\Omega\del \{e,f\}/p$ that avoids $w$.
So the existence of this path establishes that there is an edge joining a vertex of $P'-u'$ to the end-vertex of $p$ that is not equal to $u'$.
Let $h$ be such an edge.

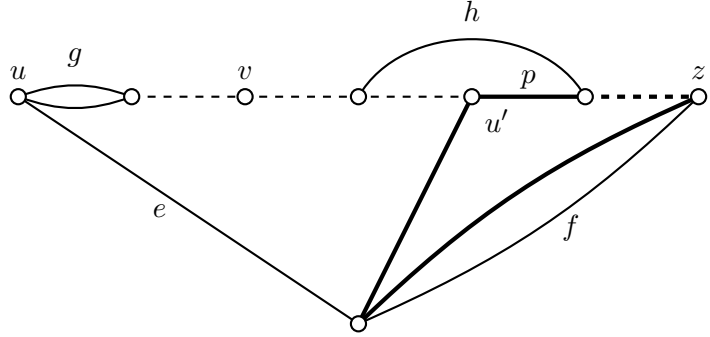
\begin{figure}[htb]
\centering

\begin{tikzpicture}[nodestyle/.style = {circle, inner sep = 0mm, minimum size = 2mm, draw = black, thick, fill = white}]

\node (tl) at (0,2) {};
\node (gr) at (1.5,2) {};
\node (v) at (3,2) {};
\node (hl) at (4.5,2) {};
\node (dl) at (6,2) {};
\node (pr) at (7.5,2) {};
\node (tr) at (9,2) {};
\node (bm) at (4.5,-1) {};

\draw[thick] (tl.center) -- (bm.center) node [pos = 0.5, left = 1.5mm] {$e$};
\draw[thick] (tr.center) edge[bend left = 10] node [pos = 0.5, right = 1.5mm] {$f$} (bm.center);
\draw[thick] (tl.center) edge[bend left = 20] node [pos = 0.5, above = 1mm] {$g$} (gr.center);
\draw[thick] (tl.center) edge[bend right = 20] (gr.center);
\draw[thick, dashed] (gr.center) -- (dl.center) {};
\draw[thick] (hl.center) edge[bend left = 60] node [pos = 0.5, above = 1mm] {$h$} (pr.center);

\draw[ultra thick] (dl.center) -- (pr.center) node [pos = 0.5, above = -0.5mm] {$p$};
\draw[ultra thick, dashed] (pr.center) -- (tr.center);
\draw[ultra thick] (tr.center) edge[bend right = 10] (bm.center);
\draw[ultra thick] (bm.center) -- (dl.center);

\node[nodestyle] at (tl) {};
\node at (tl) [above = 1mm] {$u$};
\node[nodestyle] at (tr) {};
\node at (tr) [above = 1mm] {$z$};
\node[nodestyle] at (bm) {};
\node[nodestyle] at (gr) {};
\node[nodestyle] at (dl) {};
\node at (dl) [below right = 0.5mm] {$u'$};
\node[nodestyle] at (v) {};
\node at (v) [above = 1mm] {$v$};
\node[nodestyle] at (pr) {};
\node[nodestyle] at (hl) {};

\end{tikzpicture}

\caption{The biased graph $\Omega$.
The cycle $D$ is marked in bold lines.
The horizontal edges belong to the path $P$.
Dashed lines indicate paths of unspecified length and solid lines indicate edges.
The only non-trivial parallel class of $\Omega\del\{e,f\}$ is the one containing $g$.}
\label{fig:final-case}
\end{figure}

Note that $h$ joins two non-adjacent vertices of $P$, so it is a proper chord of the bad cycle $C$.
Thus \cref{clm:proper-chord-unbalanced} implies that the cycle contained in $P\cup h$ is unbalanced.
This cycle has more than two edges.
As $\Omega\del \{e,f\}/g$ is $2$\dash connected, \cref{clm:bad-containing-x} implies that $g$ must be parallel to an edge of the cycle in $P\cup h$, which means that $g$ is parallel to $h$.
This is impossible since $h$ is incident with a vertex not in $P'$.
\end{proof}

Now we know that $D$ is a Hamiltonian cycle of $\Omega\del \{e,f\}$, so it contains all of $P$ and $Q$.
Since $\Omega$ has at least four vertices, it follows that  there are two distinct edges, call them $g$ and $g'$, that are in either $P$ or $Q$.
Note that $g$ and $g'$ are in $D$.
Now $\Omega\del \{e,f\}/g$ contains a Hamiltonian cycle, and is therefore $2$\dash connected.
The same statement is true of $\Omega\del \{e,f\}/g'$.
Therefore \cref{clm:bad-containing-x} implies that there is a unique non-trivial parallel class of $\Omega\del \{e,f\}$ and this class contains both $g$ and $g'$.
As $g$ and $g'$ are not parallel, we have reached a final contradiction in the proof of \cref{thm:main-theorem}.
\end{proof}

\section{Partial groups}

In this section, we consider gainability over a more general algebraic structure.
We will use these ideas to prove \cref{thm:main1,thm:main2}, restated below.  (Recall that a biased graph is \emph{regular} if it is $\Gamma$\dash gainable for every {non-trivial} group~$\Gamma$.)

\mainone*

Using \cref{thm:main1} together with \cref{thm:Zas81,thm:Z3-excluded-minors}, we obtain the excluded minors for regular biased graphs as an immediate consequence.

\maintwo*

Our introduction of ``partial groups'' mirrors the development of ``partial fields'' in order to attack problems on characterising matroids that are representable over a set of fields \cites{MR1390574,MR2718674,MR2563513}.
In particular, for any finite set of fields $\mathcal{F}$, there exists a partial field $\mathbb{P}$ such that a matroid is representable over each field in $\mathcal{F}$ if and only if it is representable over $\mathbb{P}$.
We prove an analogous result for partial groups, as \cref{fundamentalthmforpartialgroups}.

\subsection{An introduction to partial groups}

\begin{definition}\label{def:partial-group}
A \emph{partial group} is a pair $\bL=(\Gamma,X)$ where $\Gamma$ is a group with identity $\identity{\Gamma}$, and $X\subseteq \Gamma$ is a set satisfying:
\begin{enumerate}[label = \textup{(\roman*)}]
    \item\label{pg:identity} $\identity{\Gamma}\in X$,
    \item\label{pg:inverse} if $a\in X$, then $a^{-1}\in X$,
    \item\label{pg:conjugate} if $a\in X$ and $g\in \Gamma$, then $gag^{-1}\in X$.
\end{enumerate}
Note that when $\Gamma$ is abelian, condition (iii) is satisfied trivially.
We say $a$ is an \emph{element of $\bL$} when $a\in X$, and denote this by writing $a\in\bL$.
We say that $\identity{\Gamma}$ is the \emph{identity} of $\bL$, and we also write this element as $\identity{\bL}$.
\end{definition}

Let $\bL=(\Gamma,X)$ be a partial group, and let $\gamma$ be a $\Gamma$\dash gaining of the biased graph $\Omega$.
If the gain of each closed walk is in $X$, then we say that $\gamma$ is a \emph{$\bL$\dash gaining} of $\Omega$.
If $\Omega$ has an $\bL$\dash gaining, then we say it is 
\emph{$\bL$\dash gainable}.
Note that the gains of simple closed walks of a cycle are inverses when they are the reverse of one another, and conjugates when they are cyclic shifts of one another (but are in the same direction).
So by properties~\ref{pg:inverse} and~\ref{pg:conjugate} of \cref{def:partial-group}, we have the following.
\begin{remark}
    If the gain of one simple closed walk of a cycle $C$ is in $X$, then the gains of all simple closed walks of $C$ are in $X$. Furthermore, they are either all equal to $\identity{\bL}$ or all not equal to $\identity{\bL}$.
\end{remark}
Note that when $\Gamma$ is a group, the pair $(\Gamma,\Gamma)$ is a partial group,
and a $\Gamma$\dash gaining is the same as a $(\Gamma,\Gamma)$\dash gaining. We will thus identify a group $\Gamma$ with its associated partial group $(\Gamma,\Gamma)$.

We observe that switching preserves gainability over a partial group.
\begin{remark}\label{rk:switching}
Let $\gamma(\arc{a}_1)\cdots\gamma(\arc{a}_n)\in X$ be the gain of a simple closed walk.
Switching at an internal vertex does not change this gain, while switching by $g\in \Gamma$ at the end-vertex gives the gain $g\left[\gamma(\arc{a}_1)\cdots\gamma(\arc{a}_n)\right]g^{-1}$.
Property~\ref{pg:conjugate} of \cref{def:partial-group} implies that this gain is in $X$.
Since any switching can be produced by performing a sequence of switchings at vertices, it follows that $\gamma^{\eta}$ is an $\bL$\dash gaining of the biased graph $\Omega$ whenever $\gamma$ is an $\bL$\dash gaining of $\Omega$ and $\eta$ is a switching function.
\end{remark}

\begin{lemma}\label{lem:standard_rep}
Let $\bL=(\Gamma,X)$ be a partial group.
Let $\gamma$ be an $\bL$\dash gaining of the biased graph $\Omega=(G,\cB)$. 
Let $T$ be a forest of $G$.
Then $\Omega$ has an $\bL$\dash gaining where every arc receives a gain in $\bL$ and every edge in $T$ receives the gain $\identity{\bL}$.
\end{lemma}

\begin{proof}
We observed earlier that there is a switching function $\eta$ such that $\gamma^{\eta}$ takes each edge in $T$ to $\identity{\Gamma} = \identity{\bL}$.
\cref{rk:switching} implies that the gain of any simple closed walk under the gaining $\gamma^{\eta}$ belongs to $X$.
Let \arc{a} be an arc of the edge $e$.
If $e$ is in $T$, then $\gamma^{\eta}$ takes \arc{a} to $\identity{\bL}$.
Otherwise, $\gamma^{\eta}(\arc{a})$ is equal to the gain of a simple closed walk corresponding to the unique cycle in $T\cup e$.
This implies that $\gamma^{\eta}(\arc{a})$ is in $X$.
The result follows.
\end{proof}

\begin{proposition}\label{prop:gaining-minor}
    Let $\bL$ be a partial group. The class of $\bL$\dash gainable biased graphs is closed under minors.
\end{proposition}
\begin{proof}
    Let $\gamma$ be an $\bL$\dash gaining of a biased graph $\Omega=(G,\cB)$.
    Let $A$ and $B$ be disjoint sets of edges such that $\Omega[A]$ is balanced.
    Let $T'$ be a maximal forest of $\Omega[A]$ and let $T$ be a maximal forest of $\Omega$ that contains $T'$.
\cref{lem:standard_rep} implies there is an $\bL$\dash gaining $\gamma$ of $\Omega$ where edges in $T$ receive a gain of $\identity{\bL}$.
It follows that any edge in $A$ also receives a gain of $\identity{\bL}$.
Now for a cycle $C$ in $\Omega/A\del B$, consider a cycle $C'$ in $\Omega$ corresponding to $C$; that is to say every edge of $C$ is in $C'$, and every edge of $C'-C$ is in $A$. So by taking a closed walk $W'$ on $C'$ and omitting the identity gains on elements of $A$, we have that $\gamma(W')$ is equal to the gain $C$ receives from the restriction of $\gamma$ to $\Omega/A\del B$.
By the definition of minors, $C$ is in $\cB/A\del B$ precisely when $C'$ is in $\cB$, which is when true exactly when $\gamma(W')=\identity{\Gamma}$.
Thus restricting $\gamma$ to the edges of $\Omega /A\del B$ produces an $\bL$\dash gaining for $\Omega/A\del B$.
\end{proof}

\subsection{Partial-group Homomorphisms}

\begin{definition}\label{def:homomorphism}
Let $\bL_{1}=(\Gamma_{1},X_{1})$ and $\bL_{2}=(\Gamma_{2},X_{2})$ be partial groups with identities $\identity{\bL_{1}}$ and $\identity{\bL_{2}}$ respectively, and let $\phi : X_{1} \rightarrow X_{2}$ be a function.
Then $\phi$ is a \emph{partial-group homomorphism} if it satisfies the following:
\begin{enumerate}[label = \textup{(\roman*)}]
    \item\label{hom:id} $\phi(\identity{\bL_{1}}) = \identity{\bL_{2}}$,
    \item\label{hom:nonid} if $x$ is in $X_{1}-\{\identity{\bL_{1}}\}$, then $\phi(x) \neq \identity{\bL_{2}}$,
    \item\label{hom:prod} if $a, b\in X_{1}$ are such that $ab$ is in $X_{1}$, then $\phi(ab) = \phi(a)\phi(b)$.
\end{enumerate}
\end{definition}

\begin{example}
Let $\ZZ_{2}$ be $\{1,\omega\}$ and let $\ZZ_{2}^{2}$ be $\{1,a,b,ab\}$.
The map
\[
1\mapsto 1,\quad a\mapsto \omega,\quad b\mapsto \omega
\]
is a partial-group homomorphism from $(\ZZ_{2}^{2}, \{1,a,b\})$ to $(\ZZ_{2}, \ZZ_{2})$.
\end{example}

Note that a composition of homomorphisms is also a homomorphism.
We say that partial groups $\bL_1$ and $\bL_2$ are \emph{homomorphically equivalent} if there exist partial-group homomorphisms from $\bL_1$ to $\bL_2$, and from $\bL_2$ to $\bL_1$. Note the existence of a bijective partial-group homomorphism between $\bL_1$ and $\bL_2$ is sufficient for them to be  homomorphically equivalent. However it is not necessary; indeed, one can easily construct a homomorphism from $(\ZZ_2,\ZZ_2)$ to $(\ZZ^2_2,\{1,a,b\})$, so these partial groups are homomorphically equivalent.

\begin{lemma}
\label{lem:homomorphisms}
  Let $\bL_1$ and $\bL_2$ be partial groups, and let $\phi : \bL_1 \rightarrow \bL_2$ be a partial-group homomorphism.  If a biased graph is gainable over $\bL_1$, then it is gainable over $\bL_2$.
\end{lemma}
\begin{proof}
    Suppose a biased graph $\Omega=(G,\cB)$ is gainable over $\bL_1$.
    Take some maximal forest $T$ of $\Omega$.
    By \cref{lem:standard_rep}, there is an $\bL_1$\dash gaining $\gamma_1$ of $\Omega$ where every arc receives a label in $\bL_1$ and the edges of $T$ receive the label $1_{\bL_1}$.
    We claim that $\gamma_2=\phi\circ\gamma_1$ is an $\bL_{2}$\dash gaining of $\Omega$.
    We first show that $\phi$ distributes over simple closed walks.
    \begin{claim}
        For a simple closed walk $W$ in $G$, we have $\gamma_2(W)=\phi(\gamma_1(W))$.
    \end{claim}
    \begin{proof}
        For a cycle $C$ of $G$, we proceed by induction on the number of edges in $C-T$. As $T$ is a forest, we can assume that $C$ has at least one element not in $T$.
        
        If $|C-T|=1$, then by Lemma~\ref{lem:standard_rep}, one of the arcs in $C$ is given a label in $\bL_1$ while the rest are given the label $\identity{\bL_{1}}$.
        Let $W$ be a simple closed walk contained in $C$ and let \arc{b} be an arc in $W$ such that every other arc of $W$ receives the label $\identity{\bL_{1}}$.
Then
\begin{multline*}
\gamma_{2}(W) =
\prod_{\arc{a}\in W}\gamma_2(\arc{a})=
\prod_{\arc{a}\in W}\phi(\gamma_1(\arc{a}))\\
=\phi(\gamma_{1}(\arc{b}))=
\phi\left(\prod_{\arc{a}\in W}\gamma_{1}(\arc{a})\right)=
\phi(\gamma_1(W)).
\end{multline*}
So the claim holds when $|C-T|=1$.
Suppose $|C-T|\geq 2$ and that the claim holds for all cycles $C'$ with $|C'-T|<|C-T|$.
Take distinct $e,f\in C-T$ and consider the two components of $C-e-f$. As $T$ is a maximal forest, there is a path $P$ within $T$ between two vertices $x$ and $y$ in distinct components. Note $C\cup P$ is a theta graph containing the edges $e$ and $f$.
        Let $C_e\ni e$ be the cycle in $C\cup P-f$ and let $C_f\ni f$ be the cycle in $C\cup P-f$.
        Note $C$ is equal to the symmetric difference $C_e \triangle C_f$ and that $|C_e-T|,|C_f-T|<|C-T|$ as they are missing $f$ and $e$ respectively.
        Let $W$ be the simple walk of $C$ that starts and ends at $x$, and traverses $e$ before $f$.
        Let $W_e$ be the simple walk of $C_e$ that starts at $x$ and traverses $e$ before finally taking $P$ from $y$ to $x$. Let $W_f$ be the simple walk of $C_f$ that starts by taking $P$ from $x$ to $y$, traverses $f$ and ends at $x$.
        We note that for $i\in\{1,2\}$ we have
        \[\gamma_i(W)=\prod_{\arc{a}\in W}\gamma_i(\arc{a})=\gamma_i(W_e)\gamma_i(W_f).\]
        As $\gamma_1$ is an $\bL_1$\dash gaining of $\Omega$, we have that $\gamma_1(W_e),\gamma_1(W_f)$ and $\gamma_1(W)=\gamma_1(W_e)\gamma_1(W_f)$ are in $\bL_1$, so $\phi(\gamma_1(W_e))\phi(\gamma_1(W_f))=\phi(\gamma_1(W))$.
        By the induction hypothesis, we have
        \[
   \gamma_2(W_e)\gamma_2(W_f)=\phi(\gamma_1(W_e))\phi(\gamma_1(W_f)).
        \]
        Putting all of this together, we have that
        \begin{align*}
        \gamma_2(W) &=\gamma_2(W_e)\gamma_2(W_f)\\
                    &=\phi(\gamma_1(W_e))\phi(\gamma_1(W_f))\\
                    &=\phi(\gamma_1(W_e)\gamma_1(W_f))\\
                    &=\phi(\gamma_1(W)).
        \end{align*}
        So $\phi$ distributes over closed simple walks as desired.
        \end{proof}
        We are now ready to show that $\gamma_2=\phi\circ\gamma_1$ is a gaining of $\Omega=(G,\cB)$ over $\bL_2$.
        Consider a simple closed walk $W$ of a cycle $C$ of $G$.
        By the previous claim we have that $\gamma_2(W)=\phi(\gamma_1(W))$.
        By Properties~\ref{hom:id} and~\ref{hom:nonid} of \cref{def:partial-group}, we have that $\phi(\gamma_1(W))$ is equal to $\identity{\bL_{2}}$ precisely when $\gamma_1(W)=\identity{\bL_{1}}$, which in turn is precisely when $C\in\cB$ as $\gamma_1$ is an $\bL_1$\dash gaining of $\Omega$.
        Thus $\gamma_2$ is indeed an $\bL_2$\dash gaining of $\Omega$.
\end{proof}

We note however, that a partial-group homomorphism $\phi:\bL_1\to\bL_2$ is likely not necessary for $\bL_1$\dash gainability to imply $\bL_2$\dash gainability and there may be more general forms of \cref{lem:homomorphisms}.
For comparison, see the lift theorem for partial fields~\cite{MR2563513}.  

We obtain the following as a consequence of \cref{lem:homomorphisms}.
\begin{proposition}
    Let $\bL_1$ and $\bL_2$ be partial groups that are homomorphically equivalent.
    A biased graph $\Omega$ is gainable over $\bL_1$ if and only if it is gainable over $\bL_2$.
\end{proposition}

\subsection{Partial product}

Let $\bL_{1}=(\Gamma_{1},X_{1})$ and $\bL_{2}=(\Gamma,X_{2})$ be partial groups.
We make the following definition.
\[Z = \{(\identity{\Gamma_{1}},\identity{\Gamma_{2}})\} \cup \{(x_{1},x_{2})\in X_{1}\times X_{2}\colon x_{1} \neq \identity{\Gamma_{1}} \textrm{ and } x_{2} \neq \identity{\Gamma_{2}}\}.\]
It is not difficult to check that $(\Gamma_{1} \times \Gamma_{2},Z)$ is a partial group with identity $(\identity{\Gamma_{1}},\identity{\Gamma_{2}})$.  We call this partial group the \emph{partial product} of $\bL_{1}$ and $\bL_{2}$, and denote it $\bL_{1} \otimes \bL_{2}$.

\begin{lemma}
\label{lem:products}
Let $\bL_1$ and $\bL_2$ be partial groups.
A biased graph is gainable over both $\bL_1$ and $\bL_2$ if and only if it is gainable over the partial group $\bL_1\otimes \bL_2$.
\end{lemma}
\begin{proof}
    We start with the backward direction. Consider the projection maps
    \[\phi_1:\bL_1\otimes \bL_2\to\bL_1 \quad\text{and}\quad \phi_2:\bL_1\otimes \bL_2\to\bL_2,\]
where for $(x_{1},x_{2})\in\bL_1\otimes \bL_2$ we have $\phi_1(x_{1},x_{2})=x_{1}$ and $\phi_2(x_{1},x_{2})=x_{2}$.
    It is straightforward to verify that these are both partial-group homomorphisms.
    By \cref{lem:homomorphisms}, for any $\bL_1\otimes \bL_2$\dash gaining $\gamma$, for $i\in\{1,2\}$, we have $\phi_i\circ\gamma$ is an $\bL_i$\dash gaining.
    This proves the backward direction.

    We now consider the forward direction.
    Say the biased graph $\Omega=(G,\cB)$ has $\bL_1$\dash gaining $\gamma_1$ and $\bL_2$\dash gaining $\gamma_2$.
    We first normalise both along a spanning forest $T$.
    Note if $\arc{a}$ is an arc of $T$ or the unique circuit in $T\cup\{\arc{a}\}$ is balanced then both $\gamma_1(\arc{a})=\identity{\bL_1}$ and $\gamma_2(\arc{a})=\identity{\bL_2}$,
    while otherwise $\gamma_1(\arc{a})\in\bL_1-\{\identity{\bL_1}\}$ and $\gamma_2(\arc{a})\in\bL_2-\{\identity{\bL_2}\}$; so we observe that $(\gamma_1(\arc{a}),\gamma_2(\arc{a}))\in \bL_1\otimes \bL_2$.
    Consider the map $\gamma:A(G) \to\bL_1\otimes \bL_2$ where for an arc $\arc{a}$, we have $\gamma(\arc{a})=(\gamma_1(\arc{a}),\gamma_2(\arc{a}))$, which by the above is indeed in $\bL_1\otimes \bL_2$.
    Consider a simple closed walk $W$ of a cycle $C$ of $G$.
    Note that
    \[
    \gamma(W)=\prod_{\arc{a}\in W}(\gamma_1(\arc{a}),\gamma_2(\arc{a}))=
    \left(\prod_{\arc{a}\in W}\gamma_1(\arc{a}),\prod_{\arc{a}\in W}\gamma_2(\arc{a})\right)=
    (\gamma_1(W),\gamma_2(W)).
    \]
    As $\gamma_1$ and $\gamma_2$ are gainings, when $C$ is a balanced cycle,
    this is equal to $1_{\bL_1\otimes \bL_2}=(1_{\bL_1},1_{\bL_2})$,
    while if $C$ is not balanced, then
    \[(\gamma_1(W),\gamma_2(W))\in(\bL_1-\{\identity{\bL_1}\})\times(\bL_2-\{\identity{\bL_2}\}).\]
    Thus $\gamma(W)$ is indeed an element of $\bL_1\otimes \bL_2$ and $\gamma$ is a $\bL_1\otimes \bL_2$-gaining, as desired.
\end{proof}

The next theorem follows from \cref{lem:products}.
\begin{theorem}
\label{fundamentalthmforpartialgroups}
    For any finite set of groups $\mathcal{L}$, there exists a partial group~$\bL$ such that a biased graph is gainable over each group in $\mathcal{L}$ if and only if it is gainable over $\bL$.
\end{theorem}

\begin{lemma}\label{lem:redundant-prod}
    If there is a partial-group homomorphism $\phi:\bL_1\to\bL_2$, then the partial product $\bL_1\otimes\bL_2$ is homomorphically equivalent to $\bL_1$.
\end{lemma}

\begin{proof}
    Note we have the homomorphism $\psi_1:\bL_1\otimes\bL_2\to \bL_1$ where $\psi_1(x_{1},x_{2})=x_{1}$ for $(x_{1},x_{2})\in\bL_1\otimes\bL_2$.
    Conversely, we have the homomorphism $\psi_2:\bL_1\to\bL_1\otimes\bL_2$ where
    $\psi_2(x_{1})=(x_{1},\phi(x_{1}))$ for $x\in\bL_1$.
\end{proof}

\subsection{Regular biased graphs}
Recall that a biased graph is regular if it is $\Gamma$-gainable for every non-trivial group~$\Gamma$.
In this section, we consider
regular biased graphs.

Let $\left<\alpha\right>$ be the free group generated by the symbol~$\alpha$.
We define the \emph{regular partial group} to be $\bU_1 = (\left<\alpha\right>, \{1,\alpha,\alpha^{-1}\})$.
The following provides equivalent conditions for regularity; clearly it implies \cref{thm:main1}.
\begin{theorem}\label{thm:reg}
Let $\Omega$ be a biased graph.
The following statements are equivalent.
\begin{enumerate}[label = \textup{(\roman*)}]
\item\label{thm:reg:all_gr} $\Omega$ is regular.
\item\label{thm:reg:p_gr} $\Omega$ is $\bU_1$\dash gainable.
\item\label{thm:reg:bin_tern} $\Omega$ is gainable over $\ZZ_{2}$ and $\ZZ_{3}$.
\item\label{thm:reg:bin_nbin} $\Omega$ is gainable over $\ZZ_{2}$ and a non-trivial group with no element of order two.
\end{enumerate}
\end{theorem}
\begin{proof}
    We first show that~\ref{thm:reg:p_gr} implies~\ref{thm:reg:all_gr}.
    Let $\Gamma$ be a non-trivial group.
    That is to say, $\Gamma$ has some non-identity element $g\neq \identity{\Gamma}$.
    Consider the map $\phi:\bU_1\to \Gamma$ where $\phi(1)=\identity{\Gamma}$, $\phi(\alpha)=g$, and $\phi(\alpha^{-1})=g^{-1}$ (note $g^{-1}\neq \identity{\Gamma}$ but need not be distinct from $g$).
    It is straightforward to verify that this is a partial-group homomorphism; note $\alpha\alpha,\alpha^{-1}\alpha^{-1}\notin\bU_1$, so we do not care about the value of $\phi(\alpha)\phi(\alpha)=g^2$ or $\phi(\alpha^{-1})\phi(\alpha^{-1})=g^{-2}$.
    By \cref{lem:homomorphisms}, if a biased graph $\Omega$ is $\bU_1$\dash gainable, then it is indeed also $\Gamma$\dash gainable.

    Note that~\ref{thm:reg:all_gr} immediately implies~\ref{thm:reg:bin_tern} and~\ref{thm:reg:bin_tern} immediately implies~\ref{thm:reg:bin_nbin}.
We complete the proof by showing that~\ref{thm:reg:bin_nbin} implies~\ref{thm:reg:p_gr}.
    Recall $\ZZ_2$ is the multiplicative group $\{1,\omega\}$ with $\omega^2=1$.
    Let $\Gamma$ be a non-trivial group with no element of order two.
    Thus $\{\{g,g^{-1}\}\colon g\in \Gamma-\{\identity{\Gamma}\}\}$ is a collection of sets of size two.
    By choosing one from each of these pairs to be in $A\subseteq \Gamma-\{\identity{\Gamma}\}$ and the other to be in $B\subseteq \Gamma-\{\identity{\Gamma}\}$, we have a partition $(A,B)$ of $\Gamma-\{\identity{\Gamma}\}$ where each element and its inverse are in distinct parts.
    Let $\psi:\ZZ_2\otimes \Gamma\to\bU_1$ be the map where
    $(\identity{\ZZ_2},\identity{\Gamma})\mapsto \identity{\bU_{1}}$ and, for any $g\in \Gamma-\{\identity{\Gamma}\}$, we have $(\omega,g)\mapsto\alpha$ if $g\in A$, while $(\omega,g)\mapsto\alpha^{-1}$ if $g\in B$.
    Note Properties~\ref{hom:id} and~\ref{hom:nonid} of \cref{def:homomorphism} hold for $\psi$ by construction.
    Note if $(\omega,g)(\omega,h)=(1_{\ZZ_2},gh)$ is in $\ZZ_2\otimes \Gamma$, then $h=g^{-1}$, so 
    $g$ and $h$ are in distinct parts of $(A,B)$, in which case $\{\psi(\omega,g),\psi(\omega,h)\}=\{\alpha,\alpha^{-1}\}$, and so $\psi(\omega,g)\psi(\omega,h)=1_{\bU_1}=\psi((\omega,g)(\omega,h))$. It is now straightforward to check that Property~\ref{hom:prod} of \cref{def:homomorphism} also holds for $\psi$, and $\psi$ is thus a partial-group homomorphism.
    Suppose that a biased graph $\Omega$ is gainable over both $\ZZ_2$ and $\Gamma$.
    By Lemma~\ref{lem:products}, we have that $\Omega$ is $\ZZ_2\otimes \Gamma$\dash gainable.
    As we have shown $\psi$ to be a homomorphism, by Lemma~\ref{lem:homomorphisms}, we have that $\Omega$ is $\bU_1$\dash gainable.
    We have thus shown that the four hypotheses of our theorem are equivalent.
\end{proof}

\subsection{Near-regular biased graphs}
We say that a biased graph is \emph{near-regular} if it is gainable over all groups with size at least three.
In this section, we consider near-regular biased graphs.

Let $\left<\alpha,\beta\right>^{\ab}$ denote the free abelian group generated by distinct symbols $\alpha$ and $\beta$.
Note that because $\left<\alpha,\beta\right>^{\ab}$ is abelian, any subset that contains the identity and which is closed under taking inverses will comprise a partial group.
We define the \emph{near-regular partial group} to be
\[\bU^{\ab}_2=(\langle\alpha,\beta\rangle^{\ab},\{\identity{},\alpha,\alpha^{-1},\beta,\beta^{-1},\alpha\beta^{-1},\alpha^{-1}\beta\}).\]
The following provides equivalent conditions for near-regularity; clearly it implies \cref{thm:main3}.
\begin{theorem}\label{thm:n-reg}
Let $\Omega$ be a biased graph.
The following statements are equivalent.
\begin{enumerate}[label = \textup{(\roman*)}]
\item\label{thm:n-reg:3biggr} $\Omega$ is near-regular.
\item\label{thm:n-reg:p_gr} $\Omega$ is $\bU^{\ab}_2$\dash gainable.
\item\label{thm:n-reg:3andKlein} $\Omega$ is $\ZZ_3$\dash gainable and $\ZZ^2_2$\dash gainable.
\end{enumerate}
\end{theorem}

\begin{proof}
    We first show that~\ref{thm:n-reg:p_gr} implies~\ref{thm:n-reg:3biggr}.
    Let $\Gamma$ be a group of size at least three.
    That is to say, $\Gamma$ has distinct non-identity elements $g,h\neq 1_\Gamma$.
    Consider the map $\phi:\bU_1\to \Gamma$ where
        \[
    \begin{array}{llll}
         1_{\bU^{\ab}_2} \mapsto 1_\Gamma,\quad   & \alpha \mapsto g,         & \beta \mapsto h,& \alpha\beta^{-1} \mapsto gh^{-1}, \\
                            & \alpha^{-1}\mapsto g^{-1},\quad  & \beta^{-1}\mapsto h^{-1}, & \alpha^{-1}\beta\mapsto g^{-1}h.
    \end{array}
    \]
    It is straightforward to verify that this is a partial-group homomorphism; as $g$ and $h$ are distinct, neither $gh^{-1}$ nor $g^{-1}h$ are equal to $1_\Gamma$.
    By \cref{lem:homomorphisms}, if a biased graph $\Omega$ is $\bU^{\ab}_2$\dash gainable, then it is indeed also $\Gamma$\dash gainable.

    Note that~\ref{thm:n-reg:3biggr} immediately implies~\ref{thm:n-reg:3andKlein}.

    Finally, we show~\ref{thm:n-reg:3andKlein} implies~\ref{thm:n-reg:p_gr}.
    We write $\ZZ_3 = \langle z \mid z^3 = 1 \rangle $ and $\ZZ_2^2 = \langle x,y \mid x^2=1, y^2=1 \rangle^ \ab$.
    Let $\psi : \ZZ_3 \otimes \ZZ^2_2 \rightarrow \bU_2^\ab$ be the function where
    \[
    \begin{array}{llll}
         (1,1) \mapsto 1,\quad   & (z,x) \mapsto \alpha,         & (z,y) \mapsto \beta^{-1},\quad & (z,xy) \mapsto \alpha^{-1}\beta, \\
                            & (z^2,x) \mapsto \alpha^{-1},\quad  & (z^2,y) \mapsto \beta, & (z^2,xy) \mapsto \alpha\beta^{-1}.
    \end{array}
    \]
    It is straightforward to check that $\psi$ is a partial-group homomorphism.  
    We now suppose that a biased graph $\Omega$ is gainable over both $\ZZ_3$ and $\ZZ^2_2$.
    By Lemma~\ref{lem:products}, we have that $\Omega$ is $\ZZ_3\otimes \ZZ^2_2$\dash gainable.
    Because $\psi$ is a partial-group homomorphism, by Lemma~\ref{lem:homomorphisms}, we have that $\Omega$ is $\bU^{\ab}_2$\dash gainable.
    We have thus shown that the three hypotheses of our theorem are equivalent.
    \end{proof}

\subsection{Biased graphs gainable over \texorpdfstring{$\ZZ_3$}{Z3} and another cyclic group}

In this section, we consider when a biased graph is gainable over $\ZZ_3$ and another cyclic group $\ZZ_n$.
Say $\ZZ_3=\{1,x,x^2\}$ and $\ZZ_n=\{1,y,\dots,y^{n-1}\}$.
By Lemma~\ref{lem:redundant-prod}, if $n$ is divisible by three, then $\ZZ_3\otimes\ZZ_n$ is homomorphically equivalent to $\ZZ_3$. 

Suppose that $n$ is not divisible by three, so $\ZZ_3\times\ZZ_n\cong\ZZ_{3n}$ by Sunzi's Theorem.
By taking $\omega=(x,y)$, we observe that $\omega^i=(x^{\text{$i$ mod 3}},y^{\text{$i$ mod $n$}})$, so
\[\ZZ_3 \otimes \ZZ_n = (\langle\omega \mid \omega^{3n} = 1\rangle, \{1\}\cup\{\omega^i \mid \textrm{$3$ and $n$ do not divide $i$}\}).\]

We consider the example $n=4$.
\begin{corollary}
    A biased graph is gainable over $\ZZ_3$ and $\ZZ_4$ if and only if it is $\bS_2$\dash gainable where $\bS_2=(\langle\alpha\rangle,\{1,\alpha,\alpha^2,\alpha^{-1},\alpha^{-2}\})$.
\end{corollary}
\begin{proof}
By the above discussion,
\[
\ZZ_3 \otimes \ZZ_4 =(\langle\omega \mid \omega^{3n} = 1\rangle, \{1,\omega,\omega^2,\omega^5,\omega^7,\omega^{10},\omega^{11})\}
\]
By taking $\omega\mapsto\alpha$, $\omega^2\mapsto \beta$ and $\omega^7\mapsto\gamma$, we find that this is homomorphically equivalent to 
\[(\langle\alpha,\beta,\gamma \mid \alpha^2=\beta=\gamma^2\rangle^\ab, \{1,\alpha,\alpha^{-1},\beta,\beta^{-1},\gamma,\gamma^{-1}\}),
\]
as we only care about products within the partial group.
By taking $\alpha\mapsto\alpha$, $\beta\mapsto\alpha^2$, and $\gamma\mapsto\alpha$, we have that this is also homomorphically equivalent to $\bS_2=(\langle\alpha\rangle,\{1,\alpha,\alpha^2,\alpha^{-1},\alpha^{-2}\})$.
As $\ZZ_3\otimes\ZZ_4$ is homomorphically equivalent to $\bS_2$, the theorem follows immediately from \cref{lem:products,lem:homomorphisms}.
\end{proof}

It is easy to see that there is a partial-group homomorphism from $\mathsf{S}_2$ to any group with an element of order at least three.  Thus, the above corollary implies \cref{thm:main4}.

\end{document}